\documentclass[11pt,leqno]{amsart}%
\usepackage{amsmath}
\usepackage{amsfonts}
\usepackage{amssymb}
\usepackage{graphicx}%
\providecommand{\U}[1]{\protect\rule{.1in}{.1in}}
\newtheorem{theorem}{Theorem}[section]

\newtheorem{corollary}[theorem]{Corollary}

\newtheorem{definition}[theorem]{Definition}
\newtheorem{example}[theorem]{Example}

\newtheorem{lemma}[theorem]{Lemma}

\newtheorem{problem}[theorem]{Problem}
\newtheorem{proposition}[theorem]{Proposition}
\newtheorem{remark}[theorem]{Remark}

\numberwithin{equation}{section}
\begin{document}

\title{The Feller property on Riemannian manifolds}
\author{Stefano Pigola}
\address{Dipartimento di Fisica e Matematica\\
Universit\`a dell'Insubria - Como\\
via Valleggio 11\\
I-22100 Como, ITALY}
\email{stefano.pigola@uninsubria.it}

\author{Alberto G. Setti}
\address{Dipartimento di Fisica e Matematica
\\
Universit\`a dell'Insubria - Como\\
via Valleggio 11\\
I-22100 Como, ITALY}
\email{alberto.setti@uninsubria.it}

\maketitle

\begin{abstract}
The asymptotic behavior of the heat kernel of a Riemannian manifold gives rise
to the classical concepts of parabolicity, stochastic completeness (or
conservative property) and Feller property (or $C^{0}$-diffusion property).
Both parabolicity and stochastic completeness have been the subject of a
systematic study which led to discovering not only sharp geometric conditions
for their validity but also an incredible rich family of tools, techniques and
equivalent concepts ranging from maximum principles at infinity, function
theoretic tests (Khas'minskii criterion), comparison techniques etc... The
present paper aims to move a number of steps forward in the development of a
similar apparatus for the Feller property.

\end{abstract}

\section{Introduction\label{section introduction}}

This paper is a contribution to the theory of Riemannian manifolds satisfying
the Feller (or $C_{0}$-diffusion) property for the Laplace-Beltrami operator.
Since the appearence of the beautiful, fundamental paper by R. Azencott,
\cite{Az-bsmf}, new insights into such a theory (for the Laplace operator) are
mainly confined into some works by S.T. Yau, \cite{Yau-ind}, J. Dodziuk,
\cite{Do-indiana}, P. Li and L. Karp, \cite{LK}, E. Hsu, \cite{Hsu-annals},
\cite{Hsu-book}, E.B. Davies, \cite{Da-jam}. These papers, which have been
extended to more general classes of diffusion operators (see e.g.
\cite{Q-AnnalsProb}, \cite{Li-JMPA}) are devoted to the search of optimal
geometric conditions for a manifold to enjoy the Feller property. In fact,
with the only exception of Davies', the geometric conditions are always
subsumed to Ricci curvature lower bounds. The methods employed to reach their
results range from estimates of solutions of parabolic equations (Dodziuk,
Yau, Li-Karp) up to estimates of the probability of the Brownian motion on $M$
to be found in certain regions before a fixed time (Hsu). The probabilistic
approach, which led to the best known condition on the Ricci tensor, relies on
a result by Azencott (see also \cite{Hsu-book}) according to which \textit{$M$
is Feller if and only if, for every compact se $K$ and for every $t_{0}>0$,
the Brownian motion $X_{t}$ of $M$ issuing from $x_{0}$ enters $K$ before the
time $t_{0}$ with a probability that tends to zero as $x_{0}\rightarrow\infty
$.} \bigskip

Our point of view will be completely deterministic and, although parabolic
equations will play a key role in a number of crucial sections, it will
most often depend on elliptic equation theory.\bigskip

Beside, and closely related, to the Feller property one has the notions of
parabolicity and stochastic completeness. Recall that $M$ is said to be
parabolic if every bounded above subharmonic function must be constant.
Equivalently, the\ (negative definite) Laplace-Beltrami operator\ $\Delta$ of
$M$ does not possesses a minimal, positive Green kernel. From the
probabilistic viewpoint, $M$ is parabolic if the Brownian motion $X_{t}$
enters infinitely many times a fixed compact set, with positive probability
(recurrence). We also recall that $M$ is conservative or stochastically
complete if, for some (hence any) constant $\lambda>0,$ every bounded,
positive $\lambda$-subharmonic function on $M$ must be identically equal to
$0$. Here, $\lambda$-subharmonic means that $\Delta u\geq\lambda u$.
Equivalently, $M$ has the conservative property if the heat kernel of $M$ has
mass identically equal to $1$. From the probabilistic viewpoint stochastic
completeness means that, with probability $1$, the Brownian paths do not
explode to $\infty$ in a finite time. Clearly a parabolic manifold is
stochastically complete.

Both parabolicity and stochastic completeness have been the subject of a
systematic study which led to discovering not only sharp geometric conditions
for their validity (in fact, volume growth conditions) but also an incredible
rich family of tools, techniques and equivalent concepts ranging from maximum
principles at infinity, function theoretic tests (Khas'minskii criterion),
comparison techniques etc... The interested reader can consult e.g. the
excellent survey paper by A. Grigor'yan, \cite{Gr-bams}. See also
\cite{PRS-memoirs}, \cite{PRS-revista} for the maximum principle
perspective.\bigskip

The present paper aims to move a number of steps forward in the development of
a similar apparatus for the Feller property. Originally we also thought we would adopt
an elliptic point of view. While, in many instances, this proves to be the most
effective approach (for instance, in obtaining comparison results, or in the treatment of ends),
in some cases it is not clear how to implement it, and we have to resort to the
parabolic point of view (for instance studying minimal surfaces, or Riemannian coverings).

To make the treatment more readable, we decide to include the proofs of some of the basic results that
are crucial in the development of the theory. Sometimes, we shall use a
somewhat different perspective and more straightforward arguments. In fact,
our attempt is to use a geometric slant from the beginning of the theory, most notably in
interpreting the one dimensional case in terms of model manifolds.

\section{Heat semigroup and the Feller property\label{section definitions}}

Hereafter, $\left(  M,\left\langle \,,\right\rangle \right)  $ will always
denote a connected, non-necessarily complete Riemannian manifold of dimension
$\dim M=m$ and without boundary. Further requirements on $M$ will be specified
when needed. The (negative definite) Laplace-Beltrami operator of $M$ is
denoted by $\Delta$. With this sign convention, if $M=\mathbb{R}$, then
$\Delta=d^{2}/dx^{2}$. Recall from the fundamental work by J. Dodziuk,
\cite{Do-indiana}, that $M$ possesses a positive, minimal heat kernel
$p_{t}\left(  x,y\right)  $, i.e., the minimal, positive solution of the
problem%
\begin{equation}
\left\{
\begin{array}
[c]{l}%
\Delta p_{t}=\dfrac{\partial p_{t}}{\partial t},\medskip\\
p_{0+}\left(  x,y\right)  =\delta_{y}\left(  x\right)  .
\end{array}
\right.  \label{1}%
\end{equation}
Minimality means that if $q_{t}\left(  x,y\right)  $ is a second, positive
solution of (\ref{1}), then $p_{t}\left(  x,y\right)  \leq q_{t}\left(
x,y\right)  $. According to Dodziuk construction, $p_{t}\left(  x,y\right)  $
is obtained as%
\[
p_{t}\left(  x,y\right)  =\lim_{n\rightarrow+\infty}p_{t}^{\Omega_{n}}\left(
x,y\right)  ,
\]
where, having fixed any smooth, relatively compact exhaustion $\Omega
_{n}\nearrow M$, $p_{t}^{\Omega_{n}}\left(  x,y\right)  $ is the solution of
the Dirichlet boundary value problem%
\[
\left\{
\begin{array}
[c]{l}%
\Delta p_{t}^{\Omega_{n}}=\dfrac{\partial p_{t}^{\Omega_{n}}}{\partial
t}\text{, on }\Omega_{n}\medskip\\
p_{t}^{\Omega_{n}}\left(  x,y\right)  =0\text{, if }x\in\partial\Omega
_{n}\text{ or }y\in\partial\Omega_{n}\medskip\\
p_{0+}^{\Omega_{n}}\left(  x,y\right)  =\delta_{y}\left(  x\right)  .
\end{array}
\right.
\]
The following properties hold:

\begin{itemize}
\item[(a)] $p_{t}\left(  x,y\right)  \geq0$ is a symmetric function of $x$ and
$y$.

\item[(b)] $\int_{M}p_{t}\left(  x,z\right)  p_{s}\left(  z,y\right)
dz =p_{t+s}\left(  x,y\right)  $, $\forall t,s>0$,
and $\forall x,y\in M$

\item[(c)] $\int_{M}p_{t}\left(  x,y\right)  dy\leq1$, $\forall t>0$ and
$\forall x\in M$.

\item[(d)] For every bounded continuous function $u$ on $M$
set
\begin{equation*}
P_{t} u(x) = \int_{M}p_{t}\left(  x,y\right)  u\left(  y\right)
dy,
\end{equation*}
then $P_t u(x)$ satisfies the heat equation on $M\times (0,+\infty)$. Moreover, by (b) and
(c), $P_t$ extends to a contraction semigroup on every $L^p$,
called the heat semigroup of $M$.
\end{itemize}

From the probabilistic viewpoint, $p_{t}\left(  x,y\right)  $ represents the
transition probability density of the Brownian motion $X_{t}$ of $M$. In this
respect, property $(c)$ stated above means that $X_{t}$ is, in general,
sub-Markovian. In case the equality sign holds for some (hence any) $t>0$ and
$x\in M$ we say that $M$ is stochastically complete.\bigskip

Set $C_{0}\left(  M\right)  =\left\{  u:M\rightarrow\mathbb{R}\text{
continuous}:u\left(  x\right)  \rightarrow0\text{, as }x\rightarrow
\infty\right\}  .$

\begin{definition}
The Riemannian manifold $\left(  M,\left\langle ,\right\rangle \right)  $ is
said to satisfies the $C_{0}$-diffusion property or, equivalently, it
satisfies the Feller property, if%
\begin{equation}
P_{t} u(x)
\rightarrow0\text{, as }x\rightarrow
+\infty\label{2}%
\end{equation}
for every $u\in C_{0}\left(  M\right)  $.
\end{definition}

Using property $(c)$ of the heat kernel and a cut-off argument one can easily
prove the following

\begin{lemma}
\label{lemma_compsupp}Assume that $M$ is geodesically complete. Then $M$ is
Feller if and only if (\ref{2}) holds for every non-negative function $u\in
C_{c}\left(  M\right)  $.
\end{lemma}

\begin{proof}
Indeed, suppose that (\ref{2}) holds for every $0\leq u\in C_{c}\left(
M\right)  $. Let $v\in C_{0}\left(  M\right)  $ and define $v_{\pm}\left(
x\right)  =\max\left\{  0,\pm v\left(  x\right)  \right\}  \in C_{0}\left(
M\right)  $ so that $v\left(  x\right)  =v_{+}\left(  x\right)  -v_{-}\left(
x\right)  $. For every $R>0$, fix a cut off function $0\leq\xi_{R}\leq1$
satisfying $\xi_{R}=1$ on $B_{R}\left(  o\right)  $, and $\xi_{R}=0$ on
$M\backslash B_{2R}$. Then,%
\begin{align*}
(P_tv_{\pm})\left(  x\right) &  = (P_t \left(  \xi_{R}v_{\pm}\right))(x) +
(P_t \left(  (1-\xi_{R}) v_{\pm}\right))(x)\\
& \leq
(P_t \left(  \xi_{R}v_{\pm}\right))(x) +\sup_{M\backslash
B_{R}\left(  o\right)  }v_{\pm}.
\end{align*}
Since $0\leq\xi_{R}v_{\pm}\in C_{c}\left(  M\right)  $, letting $x\rightarrow\infty
$, we deduce%
\[
\lim_{x\rightarrow\infty}
(P_t v_\pm) (x)
\leq\sup_{M\backslash B_{R}\left( o\right)  }v_{\pm}.
\]
Since this inequality holds for every $R>0$, letting $R\rightarrow+\infty$ we
conclude%
\[
\lim_{x\rightarrow\infty}
(P_t v_\pm) (x) = 0
\]
The converse is obvious.
\end{proof}

As a trivial consequence of Lemma \ref{lemma_compsupp}, we point out the
following interesting characterization.

\begin{corollary}
\label{cor_compsupp}The geodesically complete manifold $M$ is Feller if and
only if, for some $p\in M$ and for every $R>0$,%
\[
\int_{B_{R}\left(  p\right)  }p_{t}\left(  x,y\right)  dy
\rightarrow0\text{, as }x\rightarrow+\infty.
\]

\end{corollary}

\section{Elliptic exterior boundary value problems and the Feller property\label{section exterior}}

This section is crucial for the development of most parts of the paper.
As we will recall below, following Azencott, the Feller
property can be characterized in terms of asymptotic properties of solutions
of exterior boundary value problems. In this direction, a basic step is
represented by the following

\begin{theorem}
\label{th_ext1}Let $\Omega\subset\subset M$ be a smooth open set and let
$q:M\rightarrow\mathbb{R}$ be a smooth function, $q\left(  x\right)
\geq0$.
Then, the problem%
\begin{equation}
\left\{
\begin{array}
[c]{l}%
\Delta h=q\left(  x\right)  h\text{ on }M\backslash\bar{\Omega}\\
h=1\text{ on }\partial\Omega\\
h>0\text{ on }M\backslash\Omega,
\end{array}
\right.  \label{ext1}%
\end{equation}
has a (unique) minimal smooth solution $h:M\backslash\Omega\rightarrow
\mathbb{R}$. Necessarily, $0<h\left(  x\right)  \leq1$. Furthermore, $h$ is
given by%
\begin{equation}
h\left(  x\right)  =\lim_{n\rightarrow+\infty}h_{n}\left(  x\right)  ,
\label{ext2}%
\end{equation}
where, for any chosen \ relatively compact, smooth exhaustion $\Omega
_{n}\nearrow M$ with $\Omega\subset\subset\Omega_{1}$, $h_{n}$ is the solution
of the boundary value problem%
\begin{equation}
\left\{
\begin{array}
[c]{l}%
\Delta h_{n}=q\left(  x\right)  h_{n}\text{ on }\Omega_{n}\backslash
\bar{\Omega}\\
h_{n}=1\text{ on }\partial\Omega\\
h_{n}=0\text{ on }\partial\Omega_{n}.
\end{array}
\right.  \label{ext3}%
\end{equation}

\end{theorem}

\begin{proof}
Indeed, the sequence $h_{n}$ is increasing and $0< h_{n}<1 $ in $\Omega
_{n}\setminus\overline\Omega$ by the maximum principle. By standard arguments,
$h_{n}$ together with all its derivatives converges, locally uniformly in
$M\setminus\overline\Omega$, to a solution of (\ref{ext2}). If $\tilde h$ is
another solution of (\ref{ext2}), then, again by the maximum principle
$h_{n}\leq\tilde h$ on $\Omega_{n}\setminus\Omega$, and passing to the limit
$h\leq\tilde h$, showing that $h$ is minimal.
\end{proof}

The following result, originally due to Azencott, builds up the bridge between
the parabolic and the elliptic viewpoints. For the sake of completeness we
provide a direct proof.

\begin{theorem}
\label{th_equiv}The following are equivalent:

\begin{enumerate}
\item $M$ is Feller.

\item For some (hence any) open set $\Omega\subset\subset M$ with smooth
boundary and for some (hence any) constant $\lambda>0$, the minimal solution
$h:M\backslash\Omega\rightarrow\mathbb{R}$ of the problem%
\begin{equation}
\left\{
\begin{array}
[c]{l}%
\Delta h=\lambda h\text{ on }M\backslash\bar{\Omega}\\
h=1\text{ on }\partial\Omega\\
h>0\text{ on }M\backslash\Omega,
\end{array}
\right.  \label{equiv1}%
\end{equation}
satisfies%
\begin{equation}
h\left(  x\right)  \rightarrow0\text{, as }x\rightarrow\infty. \label{equiv2}%
\end{equation}

\end{enumerate}
\end{theorem}

\begin{proof}
Assume that $M$ is Feller, so that the heat semigroup $P_t$ maps $C_{0}(M)$ into
itself. Let $\Omega$ be a relatively compact open set with smooth boundary and
let $\lambda>0.$ We choose a continuous function $u\geq0$ with support
contained in $\Omega$ and let $u_{t} = P_{t} u (x)$ be the solution of the
heat equation with initial data $u$, so that $u_{t}>0$ on $M$ by the parabolic
maximum principle. Next set
\[
w(x) =\int_{0}^{\infty}u_{t}(x) e^{-\lambda t} dt.
\]
Note that since $P_{t}$ is contractive on $L^{\infty}$ the integral is well
defined. Moreover, $w(x)\to0 $ as $x\to\infty$. Indeed, it suffices to show
that $w(x_{n})\to0$ for every sequence $x_{n}\to\infty$. Since $M$ is assumed
to be Feller, for every $t\geq0$, $u_{t}(x_{n}) \to0$ as $n\to\infty$, and the
required conclusion follows by dominated convergence.

Differentiating under the integral we obtain
\[%
\begin{split}
\Delta w(x)  &  = \int_{0}^{+\infty} e^{-\lambda t}\Delta u_{t} = \int
_{0}^{+\infty} e^{-\lambda t}\partial_{t} u_{t}\\
&  = -u(x) + \int_{0}^{+\infty} \lambda e^{-\lambda t} u_{t} = -u(x) + \lambda
w(x),
\end{split}
\]
so that $w$ satisfies
\[
\Delta w \leq \lambda w \quad\text{ in } M\setminus\Omega.
\]
Since $C=\inf_{\partial\Omega} w>0$, $v=C^{-1} w \geq1$ on $\partial\Omega$,
$v(x) \to0$ as $x\to\infty$. It follows that if $h_{n}$ is a sequence as in
Theorem~\ref{th_ext1}, $1=h_{n}\leq v$ on $\partial\Omega$, and $0=h_{n}<v$ on
$\partial\Omega_{n}$, so $h_{n}\leq v$ on $\Omega_{n}\setminus\Omega$, and
passing to the limit $0<h\leq v$ on $M\setminus\Omega$. Since $v$ tends to
zero as $x\to\infty$.

For the converse, assume that for a given relatively compact open set $\Omega$
and $\lambda>0$, the minimal solution $h$ of (\ref{equiv1}) satisfies
$h(x)\to0$ as $x\to\infty.$ As noted above, in order to verify that $M$ is
Feller it suffices to show that $P_{t}$ maps non-negative compactly supported
functions into $C_{0}(M)$. So let $u$ be such a function. We consider an
exhaustion $\Omega_{n}$ of $M$ with relatively compact domains with smooth
boundary such that $\Omega\cup\text{supp}\, u\subset\subset\Omega_{1}$, and
let $p_{t}^{\Omega_{n}}$ be the Dirichlet heat kernel of $\Omega_{n},$
$p_{t}^{\Omega_{n}}(x,y)\nearrow p_{t}(x,y)$, and therefore $u_{n,t}=
P_{t}^{\Omega_{n}} u\nearrow u_{t}=P_{t} u$. Moreover, since $p_{t}^{\Omega
_{n}}$ vanishes if either $x$ or $y$ lie on $\partial\Omega_{n}$ $u_{n,t}=0$
on $\partial\Omega_{n}\times[0,+\infty)$, and since the initial datum vanishes
outside $\Omega_{1}$, for every $n>1$, $u_{n,0}=0$ in $\Omega_{n}%
\setminus\Omega_{1}$.

Now we fix $t>0$. Since $h$ is strictly positive on $M\setminus\Omega$, there
exists a constant $C$ such that $Ch(x)\geq u_{s}(x)\geq u_{n,s}(x)$ for every
$x\in\partial\Omega_{1}$ and $s\in\lbrack0,t]$. It follows that the function
$v_{t}=Ch(x)e^{\lambda t}$ is a solution of the heat equation on
$M\setminus\Omega$ which satisfies $v_{s}\geq u_{n,s}$ on $(\Omega
_{n}\setminus\Omega_{1})\times\{0\}\cup\partial(\Omega_{n}\setminus\Omega
_{1})\times\lbrack0,t]$. By the parabolic maximum principle $v_{t}(x)\geq
u_{n,t}(x)$ in $\Omega_{n}\setminus\Omega_{1}$, and, letting $n\rightarrow
\infty$, $u_{t}(x)\leq Ce^{\lambda x}h(x)$. Since $h(x)\rightarrow0$ as
$x\rightarrow\infty$ we conclude that so does $u_{t}$, as required.
\end{proof}

\begin{remark}
\label{rmk-uniqueness}
It is worth pointing out that the elliptic characterization of the Feller
property involves the minimal solution to problem (\ref{equiv1}) and
not a generic solution. In fact, even on Feller manifolds, there could exist
infinitely many positive solutions which are asymptotically non-zero. This
fact is easily verified on a model manifold, and will be seen in Section~\ref{section monotonicity}.
On the other hand we will see in Section~\ref{section comparison} that on a stochastically complete
manifold the minimal solution is  the only one bounded positive solution.
\end{remark}

\section{Feller property on rotationally symmetric
manifolds\label{section models}}
We recall the following

\begin{definition}
\label{def_model}Let $g:\mathbb{R}\rightarrow\mathbb{R}$ be a smooth, odd
function satisfying $g^{\prime}\left(  0\right)  =1$, $g\left(  r\right)  >0$
for $r>0$. A (complete, non-compact) model manifold with warping function $g$
is the $m$-dimensional Riemannian manifold%
\begin{equation}
M_{g}^m=\left(  [0,+\infty)\times \mathbb{S}^{m-1},dr^{2}+g\left(  r\right)  ^{2}%
d\theta^{2}\right)  \label{ext4}%
\end{equation}
where $d\theta^{2}$ stands for the standard metric on the $\left(  m-1\right)
$-dimensional sphere $\mathbb{S}^{m-1}$. We refer to the origin $o\in M_{g}$ as the
pole of the model. The $r$-coordinate in the polar decomposition of
the metric represents the distance from $o$.
\end{definition}

It is well known that necessary and sufficient conditions for $M_{g}^{m}$ to
be parabolic or stochastically complete are expressed in terms of the solely
warped function $g$; see e.g. \cite{Gr-bams} and references therein. More
precisely, the model manifold $M_{g}^{m}$ is parabolic if and only if%
\[
\frac{1}{g^{m-1}}\notin L^{1}\left(  +\infty\right)
\]
whereas $M_{g}^{m}$ is stochastically complete if and only if%
\[
\frac{\int_{0}^{r}g^{m-1}\left(  t\right)  dt}{g^{m-1}\left(  r\right)
}\notin L^{1}\left(  +\infty\right)  .
\]
This section aims to provide a similar characterization for the validity of
the Feller property on $M_{g}^{m}$, \ thus completing the picture.

The next result will enable us to use model manifolds as test and comparison
spaces for the validity of the Feller property.

\begin{theorem}
\label{th_ext2}Let $M_{g}^{m}$ be an $m$-dimensional model manifold with
warping function $g$. Let $q\left(  r\left(  x\right)  \right)  \geq0$ be a
smooth, rotationally symmetric function. Let $h$ be the minimal solution of
the problem%
\begin{equation}
\left\{
\begin{array}
[c]{l}%
\Delta h=q\left(  r\right)  h\text{ on }M_{g}\backslash B_{R_{0}}\\
h=1\text{ on }\partial B_{R_{0}}\\
h>0\text{ on }M_{g}\backslash B_{R_{0}},
\end{array}
\right.  \label{ext5}%
\end{equation}
where $B_{R_{0}}$ is the metric ball of radius $R_{0}>0$, centered at the pole
of $M_{g}$. Then $h$ is rotationally symmetric.
\end{theorem}

\begin{proof}
Let $R_{n}$ be an increasing sequence with $R_{1}>R_{0}$, and for every $n$
let $h_{n}$ be the solution of the problem
\begin{equation}%
\begin{cases}
\Delta h_{n}=q(r(x))h_{n}\text{ on }\,B_{R_{n}}\setminus\overline{B}_{R_{0}} &
\\
h_{n}=1\text{ on }\partial B_{R_{0}} & \\
h_{n}=0\text{ on }\partial B_{R_{n}}, &
\end{cases}
\label{ext6}%
\end{equation}
so that $h=\lim_{n}h_{n}$. By the maximum principle, the solution to
(\ref{ext6}) is unique, and since coefficients and boundary data are
rotationally symmetric, so is $h_{n}$. Passing to the limit we conclude that
$h$ is rotationally symmetric.
\end{proof}

Combining Theorem \ref{th_equiv} with Theorem \ref{th_ext2}, and recalling that
the Laplacian of a radial function $u\left(  r\left( x\right)  \right)  $
on $M_{g}^{m}$ is given by
\[
\Delta u=u^{\prime\prime}+\left(  m-1\right)  \frac{g^{\prime}}{g}u^{\prime},%
\]
we immediately deduce the following important

\begin{corollary}
\label{cor_equiv}Let $M_{g}^{m}$ be an $m$-dimensional model manifold with
warping function $g$. Then $M_{g}^{m}$ is Feller if and only if, for some
(hence any) $R_{0}>0$, the minimal solution of the following o.d.e. problem%
\begin{equation}
\left\{
\begin{array}
[c]{l}%
h^{\prime\prime}+\left(  m-1\right)  \dfrac{g^{\prime}}{g}h^{\prime}=\lambda
h\text{, on }[R_{0},+\infty)\\
h\left(  R_{0}\right)  =1\\
h\left(  r\right)  >0\text{ on }[R_{0},+\infty),
\end{array}
\right.  \label{equiv3}%
\end{equation}
satisfies%
\begin{equation}
\lim_{r\rightarrow+\infty}h\left(  r\right)  =0. \label{equiv4}%
\end{equation}
\end{corollary}

In his fundamental paper, \cite{Az-bsmf}, Azencott gave necessary and
sufficient conditions for a $1$-dimensional diffusion to satisfy the Feller
property. These conditions concern with the coefficients of the corresponding
diffusion operator. On the other hand, in Corollary \ref{cor_equiv}, we showed
that the Feller property on a model manifold $M_{g}^{m}$ can be reduced to
that of a special $1$-dimensional diffusion. Therefore, we are able to give
the following geometric interpretation of Azencott result.

\begin{theorem}
\label{th_model}Let $M^m_{g}$ be an $m$-dimensional model manifold with warping
function $g$. Then $M^m_{g}$ is Feller if and only if either%
\begin{equation}
\frac{1}{g^{m-1}\left(  r\right)  }\in L^{1}\left(  +\infty\right)
\label{model1}%
\end{equation}
or%
\begin{equation}
\text{(i) }\frac{1}{g^{m-1}\left(  r\right)  }\notin L^{1}\left(
+\infty\right)  \text{\qquad and\qquad(ii) }\frac{\int_{r}^{+\infty}%
g^{m-1}\left(  t\right)  dt}{g^{m-1}\left(  r\right)  }\notin L^{1}\left(
+\infty\right)  . \label{model2}%
\end{equation}

\end{theorem}

\begin{remark}
In case $g^{m-1}\notin L^{1}\left(  +\infty\right)  $ condition (\ref{model2})
(ii) has to be understood as trivially satisfied.
\end{remark}

For the sake of completeness, we include a proof of Theorem \ref{th_model}
which is clearly modeled on Azencott arguments but it is somewhat more direct.

\begin{proof}
Assume first the validity of (\ref{model1}). For every $n\in\mathbb{N}$,
consider the function%
\begin{equation}
G_{n}\left(  r\right)  =\int_{r}^{n}\frac{1}{g^{m-1}\left(  t\right)  }dt
\label{model2.1}%
\end{equation}
and note that $u_{n}\left(  r\right)  =G_{n}\left(  r\right)  /G_{n}\left(
1\right)  $ solves the Dirichlet problem%
\begin{equation}
\left\{
\begin{array}
[c]{ll}%
\Delta u_{n}=0 & \text{on }B_{n}\left(  0\right)  \backslash B_{1}\left(
0\right) \\
u_{n}=1 & \text{on }\partial B_{1}\left(  0\right) \\
u_{n}=0 & \text{on }\partial B_{n}\left(  0\right)  .
\end{array}
\right.  \label{model2.2}%
\end{equation}
Let $\lambda>0$ be fixed and let $h_{n}\left(  r\right)  $ be the
(rotationally symmetric) solution of%
\begin{equation}
\left\{
\begin{array}
[c]{ll}%
\Delta h_{n}=\lambda h_{n} & \text{on }B_{n}\left(  0\right)  \backslash
B_{1}\left(  0\right) \\
h=1 & \text{on }\partial B_{1}\left(  0\right) \\
h=0 & \text{on }\partial B_{n}.
\end{array}
\right.  \label{model2.3}%
\end{equation}
By the maximum principle%
\begin{equation}
h_{n}\left(  r\right)  \leq u_{n}\left(  r\right)  \text{, on }B_{n}\left(
0\right)  \backslash B_{1}\left(  0\right)  . \label{model2.4}%
\end{equation}
Letting $n\rightarrow+\infty$ we deduce that the minimal, positive solution
$h\left(  r\right)  $ of%
\begin{equation}
\left\{
\begin{array}
[c]{ll}%
\Delta h=\lambda h & \text{on }M_{g}^{m}\backslash B_{1}\left(  0\right) \\
h=1 & \text{on }\partial B_{1}\left(  0\right)
\end{array}
\right.  \label{model2.5}%
\end{equation}
satisfies%
\begin{equation}
h\left(  r\right)  \leq c\int_{r}^{+\infty}\frac{1}{g^{m-1}\left(  t\right)
}dt, \label{model2.6}%
\end{equation}
with $c=\int_{1}^{+\infty}g^{1-m}\left(  t\right)  dt>0$. It follows that
$h\left(  r\right)  \rightarrow0$ as $r\rightarrow+\infty$ proving that
$M_{g}^{m}$ is Feller.

Suppose now that conditions (\ref{model2}) (i) and (ii)\ are met and, as
above, let $h$ be the minimal positive solution of (\ref{model2.5}).
Explicitly, this means that $h\left(  r\right)  \geq0$ satisfies%
\begin{equation}
\left\{
\begin{array}
[c]{l}%
\left(  g^{m-1}h^{\prime}\right)  ^{\prime}=\lambda g^{m-1}h\text{ on }\left(
1,+\infty\right) \\
h\left(  1\right)  =1\text{.}%
\end{array}
\right.  \label{model2.7}%
\end{equation}
Note that, in particular, $g^{m-1}h^{\prime}$ is increasing. On the other
hand, by Lemma \ref{lemma_decreasing2} of the previous section, $h^{\prime
}\left(  r\right)  <0$ on $(1,+\infty)$ and, therefore,
\begin{equation}
g^{m-1}\left(  r\right)  h^{\prime}\left(  r\right)  \rightarrow b\leq0\text{,
as }r\rightarrow+\infty. \label{model2.8}%
\end{equation}
We claim that, in fact, $b=0$. Indeed, suppose the contrary. Then, having
fixed $\varepsilon>0$ satisfying $b+\varepsilon<0$, we can choose $r_{0}>>1$
such that $-g^{m-1}\left(  r\right)  h^{\prime}\left(  r\right)  \geq-\left(
b+\varepsilon\right)  >0$, on $[r_{0},+\infty)$. Whence, integrating on
$[r_{0},+\infty]$ yields%
\begin{equation}
h\left(  r_{0}\right)  \geq-\lim_{r\rightarrow+\infty}h\left(  r\right)
+h\left(  r_{0}\right)  \geq-\left(  b+\varepsilon\right)  \int_{r_{0}%
}^{+\infty}\frac{1}{g^{m-1}\left(  t\right)  }dt, \label{model2.9}%
\end{equation}
which contradicts (\ref{model2}) (i). This proves the claim. Keeping in mind
this fact, we now integrate (\ref{model2.7}) on $[r,+\infty)$ and we get%
\begin{align}
-g^{m-1}\left(  r\right)  h^{\prime}\left(  r\right)   &  =\lambda\int
_{r}^{+\infty}g^{m-1}\left(  t\right)  h\left(  t\right)  dt \label{model2.10}%
\\
&  \geq\lambda\lim_{t\rightarrow+\infty}h\left(  t\right)  \int_{r}^{+\infty
}g^{m-1}\left(  t\right)  dt.\nonumber
\end{align}
Accordingly, if $g^{m-1}\notin L^{1}\left(  +\infty\right)  $ we necessarily
have $\lim_{t\rightarrow+\infty}h\left(  t\right)  =0$ and $M_{g}^{m}$ is
Feller. On the other hand, if $g^{m-1}\in L^{1}\left(  +\infty\right)  ,$
integrating (\ref{model2.10}) once more we deduce%
\begin{align}
h\left(  1\right)   &  \geq-\lim_{t\rightarrow+\infty}h\left(  t\right)
+h\left(  1\right) \label{model2.11}\\
&  \geq\lambda\lim_{t\rightarrow+\infty}h\left(  t\right)  \int_{1}^{+\infty
}\frac{\int_{r}^{+\infty}g^{m-1}\left(  t\right)  dt}{g^{m-1}\left(  r\right)
}dr.\nonumber
\end{align}
Because of (\ref{model2}) (ii), this latter forces $\lim_{t\rightarrow+\infty
}h\left(  t\right)  =0$ and $M_{g}^{m}$ is again Feller.

Conversely, we now suppose that the model $M_{g}^{m}$ is Feller, we assume
that condition (\ref{model1}) is not satisfied and we prove the validity of
(\ref{model2}) (ii). If $g^{m-1}\notin L^{1}\left(  +\infty\right)  $ then
there is nothing to prove. Otherwise, we note that, as above, $g^{m-1}\left(
r\right)  h^{\prime}\left(  r\right)  \rightarrow0$, as $r\rightarrow+\infty$.
Therefore, according to the first line in (\ref{model2.10}), we have%
\begin{align*}
-g^{m-1}\left(  r\right)  h^{\prime}\left(  r\right)   &  =\lambda\int
_{r}^{+\infty}g^{m-1}\left(  t\right)  h\left(  t\right)  dt\\
&  \leq\lambda h\left(  r\right)  \int_{r}^{+\infty}g^{m-1}\left(  t\right)
dt.
\end{align*}
and integrating this latter on $[1,+\infty]$ finally gives%
\[
+\infty=-\lim_{r\rightarrow+\infty}\log h\left(  r\right)  \leq\lambda\int
_{1}^{+\infty}\frac{\int_{r}^{+\infty}g^{m-1}\left(  t\right)  dt}%
{g^{m-1}\left(  r\right)  }dr,
\]
as desired.
\end{proof}

Recall that, on $M_{g}^{m}$,%
\begin{equation}
\mathrm{vol}\left(  \partial B_{r}\right)  =c_{m}g^{m-1}\left(  r\right)
\label{model3}%
\end{equation}
where $c_{m}$ is the volume of the Euclidean unit sphere $\mathbb{S}^{m-1}$. In
particular, by the co-area formula,%
\begin{equation}
\mathrm{vol}\left(  M_{g}^{m}\right)  -\mathrm{vol}\left(  B_{r}\right)
=c_{m}\int_{r}^{+\infty}g^{m-1}\left(  t\right)  dt. \label{model4}%
\end{equation}
Therefore Theorem \ref{th_model} can be restated more geometrically by saying
that $M_{g}^{m}$ is Feller if either%
\begin{equation}
\frac{1}{\mathrm{vol}\left(  \partial B_{r}\right)  }\in L^{1}\left(
+\infty\right)  \label{model5}%
\end{equation}
or%
\begin{equation}%
\begin{array}
[c]{ll}%
\text{(a)} & \dfrac{1}{\mathrm{vol}\left(  \partial B_{r}\right)  }\notin
L^{1}\left(  +\infty\right)  \text{ and}\\
\text{(b)} & \dfrac{\mathrm{vol}\left(  M_{g}\right)  }{\mathrm{vol}\left(
\partial B_{r}\right)  }-\dfrac{\mathrm{vol}\left(  B_{r}\right)
}{\mathrm{vol}\left(  \partial B_{r}\right)  }\notin L^{1}\left(
+\infty\right)  .
\end{array}
\label{model6}%
\end{equation}

From these considerations we deduce, in particular, the validity of the next

\begin{corollary}
\label{cor_model}Every model manifold $M_{g}^{m}$ with infinite volume has the
Feller property.
\end{corollary}

\begin{remark}
\label{rem_model}We have already recalled that a necessary and sufficient
condition for $M_{g}$ to be non-parabolic is that $g^{1-m}\in L^{1}\left(
+\infty\right)  .$ In fact, if $o$ denotes the pole of $M_{g},$ then the
function
\[
G\left(  x,o\right)  :=\int_{r\left(  x\right)  }^{+\infty}\frac{dt}%
{g^{m-1}\left(  t\right)  }%
\]
is the Green kernel with pole $o$ of the Laplace-Beltrami operator of $M_{g}$.
In passing, note also that $G\left(  x,o\right)  \rightarrow0$ as
$x\rightarrow\infty$. According to Theorem \ref{th_model}, a non-parabolic
model $M_{g}$ is Feller. Since parabolicity implies stochastic completeness,
we immediately deduce that a stochastically incomplete model is Feller. On the
other hand, neither parabolicity, nor, a fortiori, stochastic completeness
imply the Feller property. This is shown in the next example.
\end{remark}

\begin{example}
\label{ex_versus1}Having fixed $\beta>2$ and $\alpha>0$, let $g\left(
t\right)  :\mathbb{R}\rightarrow\mathbb{R}$ be any smooth, positive, odd
function satisfying $g^{\prime}\left(  0\right)  =1$ and $g\left(  r\right)
=\exp\left(  -\alpha r^{\beta}\right)  $ for $r\geq10$. Then,%
\begin{equation}
\frac{1}{g\left(  r\right)  }=\exp\left(  \alpha r^{\beta}\right)  \notin
L^{1}\left(  +\infty\right)  . \label{versus1}%
\end{equation}
Moreover,%
\begin{equation}
\frac{\int_{r}^{+\infty}g\left(  t\right)  dt}{g\left(  r\right)  }\asymp
r^{1-\beta}\in L^{1}\left(  +\infty\right)  . \label{versus2}%
\end{equation}
With this preparation, consider the $2$-dimensional model $M^2_{g}$. As observed
above, by (\ref{versus1}) $M^2_{g}$ is parabolic. On
the other hand, according to Theorem \ref{th_model}, condition (\ref{versus2})
implies that $M^2_{g}$ is not Feller.
\end{example}

\section{Monotonicity properties and non-uniqueness of
bounded solutions of the exterior problem\label{section monotonicity}}

In Theorem~\ref{th_model} we were able to
characterize the validity of the Feller property on a model manifold in terms
of minimal solutions $h$ of the o.d.e. problem (\ref{equiv3}), namely%
\begin{equation}
\left\{
\begin{array}
[c]{l}%
h^{\prime\prime}+\left(  m-1\right)  \dfrac{g^{\prime}}{g}h^{\prime}=\lambda
h\text{, on }[R_{0},+\infty)\\
h\left(  R_{0}\right)  =1\\
h\left(  r\right)  >0\text{ on }[R_{0},+\infty),
\end{array}
\right.  \tag{\ref{equiv3}}%
\end{equation}
It should be noted that $h$ enjoys interesting monotonicity properties. First
of all, we point out the following

\begin{lemma}
\label{lemma_comp}Let $R_{0},\lambda>0,$ $m\in\mathbb{N}$ and let
$g:[R_{0},+\infty)\rightarrow\left(  0,+\infty\right)  $ be a smooth function.
Assume that $u:[R_{0},+\infty)\rightarrow\lbrack0,+\infty)$ is a non-negative
solution of the inequality%
\begin{equation}
u^{\prime\prime}+\left(  m-1\right)  \frac{g^{\prime}}{g}u^{\prime}\geq\lambda
u. \label{comp4}%
\end{equation}
Suppose $u^{\prime}\left(  R_{1}\right)  \geq0$ for some $R_{1}\geq R_{0}$.
Then $u^{\prime}\left(  r\right)  \geq0$ for every $r\geq R_{1}$.
\end{lemma}

\begin{proof}
Write inequality (\ref{comp4}) in the form%
\begin{equation}
\frac{\left(  g^{m-1}u^{\prime}\right)  ^{\prime}}{g^{m-1}}\geq\lambda u.
\label{comp5}%
\end{equation}
Therefore, integrating on $[R_{1},r]$ gives%
\begin{equation}
u^{\prime}\left(  r\right)  \geq\frac{\lambda\int_{R_{1}}^{r}u\left(
t\right)  g^{m-1}\left(  t\right)  dt+g^{m-1}\left(  R_{1}\right)  u^{\prime
}\left(  R_{1}\right)  }{g^{m-1}\left(  r\right)  }\geq0, \label{comp6}%
\end{equation}
as claimed.
\end{proof}

Actually, much more can be said if we impose some further condition on the
coefficient $g$. Namely, we have the following

\begin{lemma}
\label{lemma_decreasing2}Assume that $1/g^{m-1}\notin L^{1}\left(
+\infty\right)  .$ Let $h$ be the minimal (bounded
is enough) solution of (\ref{equiv3}). Then
$h\left(  r\right)  $ is a strictly decreasing function.
\end{lemma}

We are going to prove (a more general version of) this result by using the
point of view of potential theory on model manifolds.
We need to recall the following
characterization of parabolicity due to L.V. Ahlfors (see \cite{AhSa-book} Theorem 6.C. See
also \cite{PST-arxiv} Theorem 4).

\begin{theorem}
A Riemannian manifold $\left(  M,\left\langle ,\right\rangle \right)  $ of
dimension $m\geq2$ is parabolic if and only if given a open set $G\subset M$ and
a bounded above solution $f$ of $\Delta f\geq0$ on $G$ it holds%
\begin{equation}
\sup_{G}f\leq\sup_{\partial G}f. \label{equiv8}%
\end{equation}

\end{theorem}

In particular, if we assume that $G=M\backslash\Omega$ for some $\Omega
\subset\subset M$, it turns out that the function%
\begin{equation}
r\longmapsto\sup_{\partial B_{r}}f\text{,} \quad \forall r>>1 , \label{equiv9}%
\end{equation}
is decreasing. Since the minimal solution $h$ of the problem%
\begin{equation}
\left\{
\begin{array}
[c]{l}%
\Delta h=\lambda h\text{ on }M\backslash\bar{\Omega}\\
h=1\text{ on }\partial\Omega\\
h>0\text{ on }M\backslash\Omega,
\end{array}
\right.  \tag{\ref{equiv1}}%
\end{equation}
must satisfy $0<h\leq1,$ we obtain the following conclusion.

\begin{lemma}
\label{lemma_decreasing1}Let $\left(  M,\left\langle ,\right\rangle \right)  $
be a complete, parabolic manifold and $\Omega\subset\subset M$. Let
$h:M\backslash\Omega\rightarrow\mathbb{R}$ be the minimal solution of problem
(\ref{equiv1}). Then, $\sup_{\partial B_{r}}h$ is a decreasing function of
$r>>1$.
\end{lemma}

Since for the model manifold $M_{g}^{m}$ the condition $g^{1-m}\notin L^{1}\left(
+\infty\right)  $ is equivalent to parabolicity and since the minimal solution
$h$ of (\ref{equiv3}) \ is nothing but the the minimal solution of
(\ref{equiv1}) on $M_{g}^{m}$, the (weak) monotonicity property asserted in
Lemma \ref{lemma_decreasing2} immediately follows from Lemma
\ref{lemma_decreasing1}. In order to conclude that, in fact, $h$ is strictly
decreasing, suppose by contradiction that $h^{\prime}\left(  R_{1}\right)
\geq0$ for some $R_{1}\geq R_{0}$. Then, by Lemma \ref{lemma_comp},
$h^{\prime}\left(  r\right)  \geq0$ for every $r\geq R_{1}$. On the other
hand, we have just proved that $h^{\prime}\leq0$. Therefore $h\left(
r\right)  \equiv h\left(  R_{1}\right)  $ for every $r\geq R_{1}$ which is
clearly impossible.

\par
We conclude this section showing that even on Feller manifolds, there could exist
infinitely many positive solutions which are asymptotically non-zero.

\begin{example}
\label{ex_uniqueness}Let $M_{g}^{m}$ be an $m$-dimensional non-stochastically complete model manifold.
According to Remark~\ref{rem_model} $M_{g}^{m}$ is Feller. By an equivalent characterization of stochastic completeness (see, e.g., \cite{Gr-bams}, Theorem 6.2) there exists a positive bounded function $u$ satisfying
$\Delta u= \lambda u$. By a radialization argument if necessary,  we may assume that $u$ is radial, and by scaling, we may also suppose that, given $R_0>0$, we have $u(R_0)=1$, so that $u$ solves the problem
\begin{equation}
\left\{
\begin{array}
[c]{l}%
u^{\prime\prime}+\left(  m-1\right)  \dfrac{g^{\prime}}{g}u ^{\prime}=\lambda u
\text{, on }(R_{0}\text{,}+\infty)\\
u\left(  R_{0}\right)  =1.
\end{array}
\right.  \label{m2}%
\end{equation}
Note that by the maximum principle, the subharmonic function $u$ cannot tend to zero at infinity.
Next, let $h\left(  r\left(  x\right)  \right)  $ be
the (rotationally invariant) minimal, positive solution of (\ref{m2}). Since $M_{g}^{m}$ is Feller,
$h(r)\to 0$ as $r\to\infty$, and in particular $h\neq u$.
For any fixed $\alpha$ such that $h^{\prime}\left(  R_{0}\right)<\alpha<u'\left(R_0\right) $, let $v_{\alpha}\left(
t\right)  $ be the solution of the Cauchy problem%
\begin{equation}
\left\{
\begin{array}
[c]{l}%
v_{\alpha}^{\prime\prime}+\left(  m-1\right)  \dfrac{g^{\prime}}{g}v_{\alpha
}^{\prime}=\lambda v_{\alpha}\text{, on }(R_{0}\text{,}+\infty)\\
v_{\alpha}\left(  R_{0}\right)  =1\text{, }v_{\alpha}^{\prime}\left(
R_{0}\right)  =\alpha.
\end{array}
\right.  \label{m2bis}%
\end{equation}
Since $v_{\alpha}-h$ and $u-v_\alpha$ are solutions of the Cauchy problem%
\[
\left\{
\begin{array}
[c]{l}%
w^{\prime\prime}+\left(  m-1\right)  \dfrac{g^{\prime}}{g}w^{\prime}=\lambda
w\text{, on }(R_{0}\text{,}+\infty)\\
w\left(  R_{0}\right)  =0\text{, }w^{\prime}\left(  R_0\right)  >0,
\end{array}
\right.
\]
according to Lemma \ref{lemma_comp}, they are both non-constant, increasing, hence positive, functions on $(R_{0},+\infty)$.
This means that $h< v_{\alpha}<u$ on $(R_{0}\text{,}+\infty)$. Moreover, since by assumption $h\left(
t\right)  \rightarrow0$, as $t\rightarrow+\infty$, then, necessarily,
$v_{\alpha}\left(  t\right)  \not \rightarrow 0$. It follows that, for every
such $\alpha$, the radial function $v_{\alpha
}\left(  r\left(  x\right)  \right)  $ is a
bounded positive solution of (\ref{equiv1})  which does not tend to zero at infinity.
\end{example}

\section{Comparison with model manifolds\label{section comparison}}

It is by now standard that parabolicity and stochastic completeness of a
general manifold can be deduced from those of a model manifold via curvature
comparisons. Such a result was obtained by Grigor'yan, \cite{Gr-bams}. In view
of Section \ref{section models} we can now extend the use of this comparison
technique to cover also the Feller property.

We begin with two comparison results for solutions of the exterior Dirichlet problem
which, in some sense, can be considered as \textquotedblleft
Khas'minskii-type tests \textquotedblright\ for the Feller property.
By comparison, recall that the original Khas'minskii test
for parabolicity and stochastic completeness states that $M$ is parabolic
(resp. stochastically complete) if, for some $\Omega\subset\subset M$, there
exists a superharmonic function $u>0$ on $M\backslash\Omega$ (resp. a
$\lambda$-superharmonic function $u>0$ on $M\backslash\Omega$) such that
$u\left(  x\right)  \rightarrow+\infty$ \ as $x\rightarrow\infty$. For a proof
based on maximum principle techniques, we refer the reader to
\cite{PRS-memoirs}, \cite{PRS-revista}.

Recall that, by a supersolution of the exterior
problem%
\begin{equation}
\left\{
\begin{array}
[c]{ll}%
\Delta v=\lambda v & \text{on }M\backslash\Omega\\
v=1 & \text{on }\partial\Omega,
\end{array}
\right.  \label{khas1}%
\end{equation}
we mean a function $u$ satisfying%
\[
\left\{
\begin{array}
[c]{ll}%
\Delta u\leq\lambda u & \text{on }M\backslash\Omega\\
u\geq1 & \text{on }\partial\Omega.
\end{array}
\right.
\]
A subsolution is defined similarly by reversing all the inequalities.

\begin{proposition}
\label{prop_comp}Let $\Omega$ be relatively compact open set with smooth boundary
in the Riemannian manifold $\left(  M,\left\langle ,\right\rangle \right)  $ and let
$\lambda>0$. Let $u$ be a positive supersolution of (\ref{khas1}) and let $h$
be the minimal, positive solution of (\ref{khas1}). Then%
\[
h\leq u\text{, on }M\backslash\Omega.
\]
In particular, if $u\left(  x\right)  \rightarrow0$ as $x\rightarrow\infty$
then $M$ is Feller.
\end{proposition}
\begin{proof}
Let $\left\{  \Omega_{n}\right\}  $ be a smooth exhaustion of $M$ and let $\left\{  h_{n}\right\}  $
be the corresponding sequence of functions defined in Theorem \ref{th_ext1}, with
$q\left(  x\right)  =\lambda$. Thus $h_{n}\rightarrow h$ the minimal positive
solution of (\ref{equiv1}) i.e.%
\[
\left\{
\begin{array}
[c]{l}%
\Delta h=\lambda h\text{ on }M\backslash\bar{\Omega}\\
h=1\text{ on }\partial\Omega\\
h>0\text{ on }M\backslash\Omega.
\end{array}
\right.
\]
Since $h_{n}\leq u$ on $\partial\Omega_{n}\cup\partial\Omega$, by the maximum
principle we have%
\[
h_{n}\leq u\text{ on }\Omega_{n}\backslash\Omega,
\]
and, letting $n\rightarrow+\infty,$ we deduce%
\[
h\leq u\text{, on }M\backslash\Omega.
\]
\end{proof}

As an application we get the following result which was first observed in
\cite{Az-bsmf}.

\begin{corollary}
\label{prop_parab}Let $M$ be non-parabolic with positive Green kernel
$G\left(  x,y\right)  $. Suppose that, for some (hence any) $y\in M$,
$G\left(  x,y\right)  \rightarrow0$ as $x\rightarrow\infty$. Then $M$ is Feller.
\end{corollary}

Recall that the Green kernel is related to the heat kernel of $M$ by%
\begin{equation}
G\left(  x,y\right)  =\int_{0}^{+\infty}p_{t}\left(  x,y\right)  dt.
\label{parab1}%
\end{equation}
In view of (\ref{parab1}), the assumptions of Corollary \ref{prop_parab} are
satisfied whenever we are able to provide a suitable decay estimate on the
heat kernel. However, heat kernel estimates may be used to obtain sharper
results by using directly the definition of the Feller property. This will be
exemplified in Section~\ref{section isoperimetry}.
\smallskip

In  spirit similar to that of Proposition~\ref{prop_comp}, in order to deduce that the manifold at hand is
non-Feller, one can compare with positive subsolutions of
(\ref{khas1}). Note that in this case, the conclusion holds under the additional assumption that the manifold
is stochastically complete.

\begin{theorem}
\label{th_versus_khas}Let $\left(  M,\left\langle ,\right\rangle \right)  $ be
a stochastically complete manifold, and let $v$ be a bounded, positive subsolution of
(\ref{khas1}) in $M\setminus \Omega.$
Then%
\[
h\left(  x\right)  \geq v(x).
\]
In particular, if%
\[
v\left(  x\right)  \not \rightarrow 0,\text{ as }x\rightarrow\infty,
\]
then $M$ is not Feller.
\end{theorem}
\begin{proof}
For every $\epsilon>0$, let $v_\epsilon= v-h-\epsilon,$ so
that $\Delta v_\epsilon\geq \lambda (v-h)> \lambda v_\epsilon$ on $M\setminus \overline \Omega$
and $v_\epsilon = -\epsilon $ on $\partial \Omega.$ But then $v_\epsilon ^+ =
\max\{0, v_\epsilon\}$ is bounded and satisfies  $\Delta v_\epsilon ^+ \geq\lambda v_\epsilon
^+$ on $M$. The assumption that $M$ is stochastically complete
(see, e.g. \cite{Gr-bams}, Theorem 6.2, or \cite{PRS-pams}) forces $v_\epsilon
^+\equiv 0$, that is, $u\le h+\epsilon.$ Letting $\epsilon \to 0+$
we deduce that $v\leq h$.
\end{proof}

In particular, if $v$ is a bounded positive solution of (\ref{equiv1}) then, by minimality,  $v=h$ and
we deduce the uniqueness property noted at the end of Section~\ref{section exterior}.

\begin{corollary}
\label{th_uniqueness}Let $\left(  M,\left\langle ,\right\rangle \right)  $
\ be a stochastically complete Riemannian manifold. Then, for every smooth open set
$\Omega\subset\subset M$, problem (\ref{equiv1}) \ has a unique, bounded
solution, namely, the minimal solution $h$ constructed in Theorem \ref{th_ext1}.
\end{corollary}

Clearly, in order to deduce from Theorem \ref{th_versus_khas} that $M$ is not
Feller, it is vital that the bounded subsolution $v\left(  x\right)  $ does
not converge to zero at infinity. In view of applications, we observe that
such condition can be avoided up to replacing condition $\Delta v\geq\lambda
v$ with a suitable (and in some sense more restrictive) differential
inequality. This is the content of the next

\begin{corollary}
\label{cor_versus_khas}Let $\left(  M,\left\langle ,\right\rangle \right)  $
be a stochastically complete manifold. Assume that, for some smooth open set $\Omega
\subset\subset M$, there exists a bounded solution $u_{\ast}\leq u\left(
x\right)  \leq u^{\ast}$ of the differential inequality%
\[
\Delta u\geq f\left(  u\right)  \text{, on }M\backslash\Omega,
\]
where $f\left(  t\right)  $ is a $C^{1}$ function on $u_{\ast}\leq t\leq
u^{\ast}$ such that%
\[
\text{(a) }f\left(  t\right)  >0\text{, (b) }f^{\prime}\left(  t\right)
\leq\lambda\text{,}%
\]
for some $\lambda>0$. Then $M$ is not Feller.
\end{corollary}

\begin{proof}
Let%
\[
F\left(  t\right)  =\int_{u_{\ast}}^{t}\frac{1}{f\left(  s\right)  }ds
\]
and define a new function $v\left(  x\right)  $ on $M\backslash\Omega$ by
setting%
\[
v\left(  x\right)  =e^{\lambda\left\{  F\left(  u\left(  x\right)  \right)
-F\left(  u^{\ast}\right)  \right\}  }.
\]
Clearly,%
\[
e^{-\lambda F\left(  u^{\ast}\right)  }\leq v\left(  x\right)  \leq1,
\]
and by direct computations we deduce%
\[
\Delta v\geq\lambda v.
\]
The result now follows from Theorem \ref{th_versus_khas}.
\end{proof}

As we have already noted above, the triviality of bounded positive
$\lambda$-subharmonic functions is equivalent to stochastic completeness of the underlying manifold.
As shown by Grigor'yan (see, e.g., \cite{Gr-bams}), the validity of  a similar Liouville property when $\lambda$ is
replaced by a  non-negative function  is related to parabolicity.

This suggests that a comparison result similar to Theorem~\ref{th_versus_khas} holds  for minimal solutions to the exterior problem (\ref{ext1})
\begin{equation}
\left\{
\begin{array}
[c]{l}%
\Delta h=q\left(  x\right)  h\text{ on }M\backslash\bar{\Omega}\\
h=1\text{ on }\partial\Omega\\
h>0\text{ on }M\backslash\Omega,
\end{array}
\right.  \tag{\ref{ext1}}%
\end{equation}
with  $q(x)\geq 0$ on $M$.

\begin{theorem}
\label{th_versus_khas-bis}Let $\left(  M,\left\langle ,\right\rangle \right)  $ be
a parabolic manifold and let $h$ be the minimal positive solution of (\ref{ext1}).
Assume that, for some smooth open set $\Omega
\subset\subset M$, and for some $\lambda>0$, $v$ is a positive subsolution of (\ref{ext1})
Then%
\[
h\left(  x\right)  \geq v\left(  x\right) .
\]
\end{theorem}

Clearly, Theorem~\ref{th_versus_khas-bis} yields a uniqueness result for bounded positive solutions
to the exterior problem (\ref{ext1}) companion to Corollary~\ref{th_uniqueness}.

The proof  of Theorem~\ref{th_versus_khas-bis} uses some potential theory for diffusion
operators on weighted manifolds. Given a smooth function $w$ on $M$, the
$w$-Laplacian is defined as the diffusion operator%
\[
\Delta_{w}f=e^{w}\operatorname{div}\left(  e^{-w}\nabla f\right)  .
\]
The corresponding weighted manifold $\left(  M,\left\langle ,\right\rangle
,e^{-w}d\mathrm{vol}\right)  $ is said to be $w$-parabolic if every bounded
above solution of $\Delta_{\omega}u\geq0$ must be constant. As in the usual
Riemannian case $w=0$, one has that $w$-parabolicity of a geodesically
complete manifold $\left(  M,\left\langle ,\right\rangle \right)  $ is implied
by the volume growth condition%
\[
\frac{1}{\mathrm{vol}_{w}\left(  \partial B_{r}\left(  o\right)  \right)
}\notin L^{1}\left(  +\infty\right)  ,
\]
where we have set%
\[
\mathrm{vol}_{w}\left(  \partial B_r(o)\right)  =\int_{\partial
B_r(o)}e^{-w}d\mathcal{H}^{m-1},
\]
and $\mathcal{H}^{m-1}$ is the Riemannian $(m-1)$-dimensional
Hausdorff measure.
Furthermore, one can relate the $w$-parabolicity to the vanishing of a
suitable (weighted) capacity of compact subsets. More precisely, for any fixed
closed set $C\subseteq M$, define%
\[
\mathrm{cap}_{w}\left(  C\right)  =\inf\left\{  \int_{M}\left\vert \nabla
u\right\vert ^{2}e^{-w}d\mathrm{vol}:u\in C_{c}^{\infty}\left(  M\right)
\text{ s.t. }u\geq1\text{ on }C\right\}  .
\]
Then, we have

\begin{lemma}
The weighted manifold $\left(  M,\left\langle ,\right\rangle ,e^{-w}%
d\mathrm{vol}\right)  $ is $w$-parabolic if and only if, for every compact set
$K\subset M$, $\mathrm{cap}_{w}\left(  K\right)  =0$.
\end{lemma}

Finally, one has a weighted version of the Ahlfors-type characterization:

\begin{lemma}
The wighted manifold $\left(  M,\left\langle ,\right\rangle ,e^{-w}%
d\mathrm{vol}\right)  $ is $w$-parabolic if and only if given an open set
$G\subset M$ and a bounded above solution $f$ of $\Delta_{\omega}f\geq0$ on
$G$ it holds%
\[
\sup_{G}f\leq\sup_{\partial G}f.
\]

\end{lemma}

We are now ready to give the

\begin{proof}
[Proof (of Theorem \ref{th_versus_khas-bis})]Let $h$ be the minimal solution of
problem (\ref{ext1}).
Then, the new function%
\[
f=\frac{h}{v}%
\]
satisfies%
\[
\Delta_{w}f\leq0,
\]
with%
\[
w=-\log v^{2}.
\]
Furthermore, having set%
\[
\sup_{M\backslash\Omega}v=v^{\ast}<+\infty,
\]
we have%
\[
0\leq\mathrm{cap}_{w}\left(  K\right)  \leq\left(  v^{\ast}\right)
^{2}\mathrm{cap}\left(  K\right)  =0,
\]
for every compact set $K\subset M$, proving that $\left(  M,\left\langle
,\right\rangle ,e^{-w}d\mathrm{vol}\right)  $ is $w$-parabolic. By the global
minimum principle on $M\backslash\Omega$ we deduce%
\[
f=\frac{h}{v}\geq\inf_{\partial\Omega}f\geq 1,
\]
as desired.
\end{proof}

We  now apply Proposition~\ref{prop_comp} and Theorem~\ref{th_versus_khas} to obtain
comparison results with model manifolds mentioned at the beginning of the section.

\begin{theorem}
\label{th_comparison}Let $\left(  M,\left\langle ,\right\rangle \right)  $ be
a complete, $m$-dimensional Riemannian manifold

(a) Assume that $M$ has a pole $o$. Suppose that the radial sectional
curvature with respect to $o$ satisfies%
\begin{equation}
^{M}Sec_{rad}\leq G\left(  r\left(  x\right)  \right)  \text{ on }M,
\label{comp1}%
\end{equation}
for some smooth even function $G:\mathbb{R}\rightarrow\mathbb{R}$. Let
$g:[0,+\infty)\rightarrow\lbrack0+\infty)$ be the unique solution of the
Cauchy problem%
\begin{equation}
\left\{
\begin{array}
[c]{l}%
g^{\prime\prime}+Gg=0\\
g\left(  0\right)  =0,\text{ }g^{\prime}\left(  0\right)  =1.
\end{array}
\right.  \label{comp2}%
\end{equation}
If the $m$-dimensional model $M_{g}$ has the Feller property then also $M$ is Feller.

(b) Assume that the radial Ricci curvature of $M$ satisfies%
\[
^{M}Ric\left(  \nabla r,\nabla r\right)  \geq\left(  m-1\right)  G\left(
r\left(  x\right)  \right)  ,
\]
where $r\left(  x\right)  =dist\left(  x,o\right)  $, for some $o\in M$, and
$G:\mathbb{R}\rightarrow\mathbb{R}$ is a smooth, even function. Let
$g:[0,+\infty)\rightarrow\lbrack0,+\infty)$ be the unique solution of the
problem%
\[
\left\{
\begin{array}
[c]{l}%
g^{\prime\prime}+Gg=0\\
g\left(  0\right)  =0\text{, }g^{\prime}\left(  0\right)  =1.
\end{array}
\right.
\]
If the $m$-dimensional model $M_{g}$ has finite volume and it does not satisfy
the Feller property then also $M$ is not Feller.
\end{theorem}

\begin{proof}
Let $u$ be the minimal solution of%
\begin{equation}
\left\{
\begin{array}
[c]{l}%
u^{\prime\prime}+\left(  m-1\right)  \dfrac{g^{\prime}}{g}u^{\prime}=\lambda
u\text{ on }[1,+\infty)\\
u\left(  1\right)  =1\\
u\left(  r\right)  >0\text{ on }[1,+\infty).
\end{array}
\right.  \label{comp9}%
\end{equation}
By Corollary \ref{cor_equiv}, $u\left(  r\right)  \rightarrow0$ as
$r\rightarrow+\infty.$ In particular, according to Lemma \ref{lemma_comp},
$u^{\prime}<0$ on $[1,+\infty)$.

Consider now the radial smooth function $v\left(  x\right)  =u\left(  r\left(
x\right)  \right)  $ on $M\backslash B_{1}\left(  o\right)  $. Note that%
\begin{equation}
\Delta v=u^{\prime\prime}+u^{\prime}\Delta r. \label{comp10}%
\end{equation}
Since $u^{\prime}<0$ and, by Hessian comparison,%
\begin{equation}
\Delta r\geq\left(  m-1\right)  \frac{g^{\prime}}{g}, \label{comp11}%
\end{equation}
we deduce%
\begin{equation}
\Delta v\leq u^{\prime\prime}+\left(  m-1\right)  \frac{g^{\prime}}%
{g}u^{\prime}=\Delta_{M_{g}}u=\lambda u=\lambda v. \label{comp12}%
\end{equation}
Summarizing,%
\begin{equation}
\left\{
\begin{array}
[c]{l}%
\Delta v\leq\lambda v\text{ on }M\backslash B_{1}\left(  o\right) \\
v=1\text{ on }\partial B_{1}\left(  o\right) \\
v>0\text{ on }M\backslash B_{1}\left(  o\right) \\
\lim_{r\left(  x\right)  \rightarrow+\infty}v\left(  x\right)  =0.
\end{array}
\right.  \label{comp13}%
\end{equation}
To conclude, we apply the comparison principle stated in Proposition
\ref{prop_comp} above.

(b) By assumption, $g^{m-1}\in L^{1}\left(  +\infty\right)  $ so that
$1/g^{m-1}\notin L^{1}\left(  +\infty\right)  $. Since $M_{g}$ is not Feller,
by Theorem \ref{th_model} we must have%
\[
\frac{\int_{r}^{+\infty}g^{m-1}\left(  t\right)  dt}{g^{m-1}\left(  r\right)
}\in L^{1}\left(  +\infty\right)  .
\]
Define%
\[
\alpha\left(  r\right)  =\int_{r}^{+\infty}\frac{\int_{s}^{+\infty}%
g^{m-1}\left(  t\right)  dt}{g^{m-1}\left(  s\right)  }ds.
\]
A direct computation shows that%
\[
^{M_{g}}\Delta\alpha=1.
\]
Now consider%
\[
v\left(  x\right)  =\alpha\left(  r\left(  x\right)  \right)  \text{ on
}M\backslash B_{1}.
\]
Clearly, $v$ is a positive bounded function. Since $\alpha^{\prime}\leq0$, by
Laplacian comparison we have%
\[
\Delta v\geq1.
\]
On the other hand, by the Bishop volume comparison theorem it holds%
\[
\mathrm{vol}\left(  M\right)  \leq\mathrm{vol}\left(  M_{g}\right)  <+\infty.
\]
In particular $M$ is parabolic. Applying Corollary \ref{cor_versus_khas} with
the choice $f\left(  t\right)  =1$ we conclude that $M$ is not Feller.
\end{proof}

\section{Feller property on manifolds with many ends\label{section ends}}

It is a trivial consequence of Theorem \ref{th_equiv} that Riemannian
manifolds which are isometric outside a compact set have the same behavior
with respect to the Feller property. The choice of a smooth compact set
$\Omega$ in the complete manifold $\left(  M,\left\langle ,\right\rangle
\right)  $ gives rise to a finite number of unbounded connected components,
say $E_{1},...,E_{k}$. They are called the ends of $M$ with respect to
$\Omega$. Thus, the minimal solution $h$ of (\ref{equiv1}) restricts to the
minimal solution $h_{j}$ of the same Dirichlet problem on $E_{j}$ with respect
to the compact boundary $\partial E_{j}$. Furthermore,\ $h$ tends to zero at
infinity in $M$ if and only if each function $h_{j}\left(  x\right)
\rightarrow0$ as $E_{j}\ni x\rightarrow\infty$.

This situation suggests to localize the definition of the Feller property to a
given end by saying that $E$ is Feller if, for some $\lambda>0$, the minimal
solution $g:E\rightarrow(0,1]$ of the Dirichlet problem%
\[
\left\{
\begin{array}
[c]{ll}%
\Delta g=\lambda g & \text{on \textrm{int}}\left(  E\right) \\
g=1 & \text{on }\partial E,
\end{array}
\right.
\]
satisfies $g\left(  x\right)  \rightarrow0$ as $x\rightarrow\infty$. The usual
exhausting procedure shows that $g$ actually exists.

Now, let $E_{1},...,E_{k}$ be the ends of $M$ with respect to the compact set
$\Omega$. Then, we can enlarge slightly $\Omega$ to a new compact
$\Omega^{\prime}$ which encloses a small collar neighborhood $W_{j}$ of each
$\partial E_{j}\subset E_{j}$. Since the validity of the Feller property on
$M$ is not sensitive of the chosen compact, we deduce that $M$ is Feller if
and only if each $E_{j}^{\prime}=E_{j}\backslash W_{j}$ is Feller. This
implies that, in case we have isometries $f_{j}:\partial E_{j}\rightarrow
\partial D_{j}$ \ onto the boundaries of compact Riemannian manifolds
$(D_{j},\left\langle ,\right\rangle _{D_{j}})$, then $M$ is Feller if and only
if so is each Riemannian gluing (without boundary) $E_{j}\cup_{f_{j}}D_{j}$.
Recall that, by definition, $E_{j}\cup_{f_{j}}D_{j}$ has the original metrics
outside a small bicollar neighborhood of the glued boundaries. Along the same
lines we can easily obtain that $M$ is Feller if and only if the Riemannian
double $\mathcal{D}\left(  E_{j}\right)  $ of each end $E_{j}$ has the same
property. We have thus obtained the following

\begin{proposition}
\label{prop_ends}Let $\left(  M,\left\langle ,\right\rangle \right)  $ be a
complete Riemanian manifold and let $E_{1},...,E_{k}$ be the ends of $M$ with
respect to the smooth compact domain $\Omega$. Then, the following are equivalent:

\begin{enumerate}
\item $M$ is Feller

\item Each end $E_{j}$ has the Feller property.

\item Each end with a cap $E_{j}\cup_{f_{j}}D_{j}$ (if possible) has the
Feller property.

\item The double $\mathcal{D}\left(  E_{j}\right)  $ of each end has the
Feller property.
\end{enumerate}
\end{proposition}

Using this observation, one can easily construct new Feller or non-Feller
manifolds from old ones by adding suitable ends. For instance, consider the
equidimensional, complete Riemannian manifolds $M$ and $N$ and form their
connected sum $M\#N$. This latter is Feller if and only if both $M$ and $N$
has the Feller property.

In the special case of warped products with rotational symmetry, combining Theorem \ref{th_model} with Proposition \ref{prop_ends}, we are able to obtain the following characterization.

\begin{example}
Consider the warped product of the form%
\[
\mathbb{R}\times_{f}\mathbb{S}^{m-1}=\left(  \mathbb{R}\times\mathbb{S}%
^{m-1},dr^{2}+f\left(  r\right)  ^{2}d\theta^{2}\right)
\]
where $f\left(  r\right)  >0$ is a smooth function on $\mathbb{R}$. This is a
complete manifold (without boundary) with two ends. Let $E_{1}=(1,+\infty
)\times_{f}\mathbb{S}^{m-1}$ and $E_{2}=(-\infty,1)\times_{f}\mathbb{S}^{m-1}$
be the ends of $\mathbb{R}\times_{f}\mathbb{S}^{m-1}$ with respect to the
compact domain $\Omega=[-1,1]\times\mathbb{S}^{m-1}$. Using the closed unit
disc $\mathbb{D}^{m}$ as a cap, starting from $E_{1}$ and $E_{2}$ we can
construct complete \ manifolds without boundary each isometric to a model
manifold. Precisely, $E_{1}$ gives rise to
\[
E_{g_{1}}=\left(  [0,+\infty)\times\mathbb{S}^{m-1},dr^{2}+g_{1}\left(
r\right)  ^{2}d\theta^{2}\right)
\]
where $g_{1}:[0,+\infty)\rightarrow\lbrack0,+\infty)$ satisfies $g_{1}\left(
r\right)  =r$ if $0\leq r<1-\varepsilon$ and $g_{1}\left(  r\right)  =f\left(
r\right)  $ if $r>1+\varepsilon$. Similarly,%
\[
E_{g_{2}}=\left(  [0,+\infty)\times\mathbb{S}^{m-1},dr^{2}+g_{2}\left(
r\right)  ^{2}d\theta^{2}\right)
\]
where $g_{2}:[0,+\infty)\rightarrow\lbrack0,+\infty)$ satisfies $g_{2}\left(
r\right)  =r$ if $0\leq r<1-\varepsilon$ and $g_{2}\left(  r\right)  =f\left(
-r\right)  $ if $r>1+\varepsilon$. By Proposition \ref{prop_ends},
$\mathbb{R}\times_{f}\mathbb{S}^{m-1}$ \ is Feller if and only if both
$E_{g_{1}}$ and $E_{g_{2}}$ are Feller. Since, according to Theorem
\ref{th_model}, the Feller property on model manifolds is completely
characterized by the asymptotic behavior of the warping functions, we obtain
the next
\end{example}

\begin{corollary}
\label{cor_warped}The warped product $\mathbb{R}\times
_{f}\mathbb{S}^{m-1}$ has the Feller property \ if and only if both $g\left(
t\right)  =f\left(  t\right)  ,$ $t>>1$, and $g\left(  t\right)  =f\left(
-t\right)  $, $t<<1$, satisfy either of the conditions (\ref{model1}) or
(\ref{model2}) of Theorem \ref{th_model}.
\end{corollary}

Application of this result will be given in Section \ref{section coverings}.

\section{Isoperimetry and the Feller property}

\label{section isoperimetry}

Using a general result by A. Grigor'yan, \cite{Gr-revista}, we are going to
show that a Riemannian manifold is Feller provided it satisfies a suitable
isoperimetric inequality. As a consequence we will deduce that minimal
submanifolds in Cartan-Hadamard manifolds (i.e., complete, simply connected
manifolds with non-positive sectional curvature), and in particular
Cartan-Hadamard manifolds themselves, are Feller. The latter result was proved
by Azencott, \cite{Az-bsmf}, using different methods based on comparison
arguments. Actually, in Section~\ref{section comparison} above we developed comparison
techniques which allowed us to prove the validity of the Feller property for
manifolds with a pole which are not necessarily Cartan-Hadamard.
\medskip

If $\Omega$ is a bounded domain with smooth boundary, we denote by
$\lambda_{1}(\Omega)$ the smallest Dirichlet eigenvalue of $-\Delta$ in
$\Omega$. Note that by domain monotonicity $\lambda_{1}(\Omega) $ is a
decreasing function of $\Omega$, and since a Riemannian manifold is locally
Euclidean, $\lambda_{1}(B_{r}(x_{o}))\sim c_{n} r^{-2}$ as $r\to0$.

\begin{theorem}
\label{th iso} Let $\left(  M,\left\langle ,\right\rangle \right)  $ be a
complete manifold satisfying the Faber-Krahn  isoperimetric
inequality
\begin{equation}
\lambda_{1}(\Omega) \geq\Lambda(\mathrm{vol}\Omega), \label{iso1}%
\end{equation}
for every bounded domain $\Omega\subset\subset M$, where $\Lambda$ is a
positive decreasing function such that $1/(s\Lambda(s))\in L^1(0+)$.
Let $V(t)$ be the function defined by the formula
\begin{equation}
\label{iso3}t= \int_{0}^{V(t)} \frac{ds}{s\Lambda(s)} ds,
\end{equation}
and assume that there exists $T\in(0, +\infty]$ such that
\begin{equation}
\label{iso4}\frac{tV^{\prime}(t)}{V(t)} \text{ is bounded for } t\leq2T \text{
and non-decreasing for } t>T.
\end{equation}
Then $M$ is Feller.
\end{theorem}

\begin{proof}
Indeed, it follows from \cite{Gr-revista} Theorem 5.1 that the heat kernel
$p_{t}(x,y)$ of $M$ satisfies the Gaussian estimate
\begin{equation}
\label{iso5}p_{t}(x,y)\leq\frac{C}{V(ct)} \exp\left(  -\frac{d(x,y)^{2}}%
{Dt}\right)  ,
\end{equation}
for some constants $C, c>0$ and $D>4,$ and where $V(t)$ is the function
defined in (\ref{iso3}). A straightforward application of the dominated
convergence theorem shows that for every continuous function of compact
support and for every $t>0,$
\[
P_{t}u(x) = \int_{M} p_{t}(x,y)u(y)\,d\mathrm{vol}(y) \to0 \text{ as }
\ x\to\infty,
\]
and therefore $M$ is Feller.
\end{proof}

It follows from Cheeger's inequality, see \cite{Gr-revista} Proposition 2.4,
that if the isoperimetric inequality
\[
\mathrm{vol}\left(  \partial\Omega\right)  \geq g\left(  \mathrm{vol}\left(
\Omega\right)  \right)
\]
is valid for some $g$ such that $g(s)/s$ is non-increasing, then (\ref{iso1})
holds with $\Lambda(s)=\frac{1}{4}\left(  \frac{g(s)}{s}\right)  ^{2}$. Thus,
for instance, if $M$ supports an isoperimetric inequality of the type
\begin{equation}
\mathrm{vol}(\partial\Omega)\geq C\left(  \mathrm{vol}\Omega\right)  ^{1-1/p}
\label{iso6}%
\end{equation}
for some $p\geq m$, then (\ref{iso1}) holds with $\Lambda(s)=Cs^{-2/p}$ and
the associated function $V(t)=Ct^{p/2}$ satisfies condition (\ref{iso5}). Note
that (\ref{iso6}) is equivalent to the validity of the $L^{1}$-Sobolev
inequality
\begin{equation}
||\nabla u||_{L^{1}}\geq S_{1,p}||u||_{L^{\frac{p}{p-1}}},\,\forall u\in
C_{c}^{1}(M). \label{iso7}%
\end{equation}
On the other hand, it is know from work of G. Carron, \cite{C-Spectral}, that
the Faber-Krahn inequality
\[
\lambda_{1}(\Omega)\geq C(\mathrm{vol}\Omega)^{-2/p},\quad p\geq m,p\neq2,
\]
is equivalent to the $L^{2}$-Sobolev inequality
\begin{equation}
||\nabla u||_{L^{2}}\geq S_{2,p}||u||_{L^{\frac{2p}{p-2}}},\,\forall u\in
C_{c}^{1}(M). \label{iso7bis}%
\end{equation}
From these considerations we obtain the following

\begin{corollary}
\label{cor_iso} Let $M$ be isometrically immersed into a Cartan-Hadamard
manifold. If
its mean curvature vector field $H$ satisfies
\[
||H||_{L^{m}(M)}< +\infty,
\]
then $M$ is Feller. In particular,

\begin{itemize}
\item[(a)] Every Cartan-Hadamard manifold is Feller.

\item[(b)] Every complete, minimal submanifold in a Cartan-Hadamard manifold
is Feller.
\end{itemize}
\end{corollary}

\begin{proof}
Indeed, according to \cite{HS-cpam}, there exists a constant $c_{m}$ depending
only on $m$ such that, for every bounded domain $C^{1}$ function with compact
support
\[
\left(  \int_{M} | u| ^{m/(m-1)} d\mathrm{vol} \right)  ^{(m-1)/m}\leq c_{m}
\int_{M} [|\nabla u| + |H| |u| ] d\mathrm{vol}.
\]
Since $|H|\in L^{m}(M)$, there exists a compact set $K$ such that
\[
\left(  \int_{M\setminus K} |H|^{m} d\mathrm{vol}\right)  ^{1/m}< \frac1{2
c_{m}},
\]
and applying H\"older inequality to the second summand on the RHS, we deduce
that the standard $L^{1}$ isoperimetric inequality
\begin{equation}
\label{iso8}\left(  \int_{M} | u| ^{m/(m-1)} d\mathrm{vol} \right)
^{(m-1)/m}\leq2 c_{m} \int_{M} |\nabla u| d\mathrm{vol}%
\end{equation}
holds for every $C^{1}$ function $u$ with compact support in $M\setminus K$.
A variation of a result of Carron, \cite{C-Duke} (see \cite{PST-arxiv},
Theorem 13) implies that the $L^{1}$-Sobolev inequality (\ref{iso8}) holds,
possibly with a larger constant, for every compactly supported function on
$M$. By the arguments preceding the statement of the corollary $M$ is Feller.
\end{proof}

\begin{remark}
\label{iso_rmk1} We note, in particular, that if $M$ is isometrically immersed
in a Cartan-Hadamard manifold and its mean curvature is in $L^{m}$ then $M$
has infinite volume.

We also observe that the same arguments show that if the $L^{2}$-isoperimetric
inequality (\ref{iso7bis}) holds off a compact set then it holds everywhere
and $M$ is Feller.

Finally, the above arguments show that, if $m\geq3$, then a minimal
submanifold in a Cartan-Hadamard space is non-parabolic, and its Green kernel
decays at infinity. Of course this fails in dimension two, as the example of
$\mathbb{R}^{2}$ shows.
\end{remark}

We end this section by noting that one of the most important category of
minimal surfaces is represented by those properly immersed in the ambient
space. Recall that a map between topological spaces is proper if the pre-image
of a compact set is compact. Thus, intrinsically divergent sequences cannot
accumulate at a finite point in the ambient space. In case $f:M\rightarrow N$
is a proper, minimal immersion of the complete $m$-dimensional manifold $M$
into the Cartan-Hadamard manifold $N$, the validity of the Feller property can
be also obtained using direct heat kernel comparisons. More precisely, we can
use the following result, \cite{CLY-AJM}, \cite{Ma-PAMS}.

\begin{theorem}
\label{th_heatcomparison}Let $f:M\rightarrow N$ be an $m$-dimensional,
complete, minimally immersed submanifold of the $n$-dimensional
Cartan-Hadamard manifold $N$. Let $D$ be a compact domain in $M$ with
Dirichlet heat kernel $p_{t}^{D}$. For any fixed $x\in D$, let $\mathbf{B}%
_{R}\left(  0 \right)  $ be the ball in $\mathbb{R}^{m}$ of
radius%
\[
R=\max_{y\in D}d_{N}\left(  f\left(  x\right)  ,f\left(  y\right)  \right)
,
\]
and let  $\mathbf{p}_{t}^{\mathbf{B}_{R}(0) }\left( u,v   \right) $ be the corresponding
Dirichlet heat kernel. Recalling  that
$\mathbf{p}_{t}^{\mathbf{B}_{R} }\left( 0,v   \right) =
\mathbf{p}_{t}^{\mathbf{B}_{R}(0) }\left( |v|   \right) $
depends only on $|v|$, then for every $t\geq0$ and for every $y\in D$,
\[
p_{t}^{D}\left(  x,y\right)
\leq\mathbf{p}_{t}^{\mathbf{B}_{R}(0)}\left(  d_N(f(x),f(y))\right)
.
\]
\end{theorem}

Now, let $x,y\in M$ and let $\left\{  D_{n}\right\}  $ be a smooth exhaustion
of $M$ satisfying $x,y\in D_{0}$. Then, by Theorem \ref{th_heatcomparison} and
by the parabolic comparison principle%
\begin{equation}
p_{t}^{D_{n}}\left(  x,y\right)  \leq\mathbf{p}_{t}^{\mathbf{B}_{R_{n}
}(0)  }\left(  \left(d_N(f(x),f(y))\right)\right)
\leq\mathbf{p}_{t}^{\mathbb{R}^{m}}
\left(  \left(d_N(f(x),f(y))\right)\right)  ,\label{ms1}%
\end{equation}
for every $t\geq0$. On the other hand,%
\[
p_{t}^{D_{n}}\left(  x,y\right)  \rightarrow p_{t}^{M}\left(  x,y\right)
\text{, as }n\rightarrow+\infty.
\]
Therefore, taking limits in (\ref{ms1}) we obtain that, for every $x,y\in M$
and for every $t\geq0$,%
\begin{equation}
p_{t}^{M}\left(  x,y\right)  \leq\mathbf{p}_{t}^{\mathbb{R}^{m}}\left(
\left(d_N(f(x),f(y))\right) \right)  .\label{ms2}%
\end{equation}
It follows that, for any $R>0$,%
\begin{equation}
0\leq\int_{B_{R}}p_{t}^{M}\left(  x,y\right)  d y
\leq\int_{B_{R}}\mathbf{p}_{t}^{\mathbb{R}^{m}}
\left( \left(d_N(f(x),f(y))\right) \right)  d  y.\label{ms3}%
\end{equation}
Since, by assumption, $f$ is proper we have that $f\left(  x\right)
\rightarrow\infty$ as $x\rightarrow\infty$.
Whence, taking limits in (\ref{ms3}) and using the dominated convergence
theorem on the RHS we conclude%
\[
\lim_{x\rightarrow\infty}\int_{B_{R}}p_{t}^{M}\left(  x,y\right)
dy  =0.
\]
According to Corollary  \ref{cor_compsupp}, this proves that $M$ is Feller.

\section{Feller property and covering spaces\label{section coverings}}

In this section we address the following

\begin{problem}
\label{question_cov}Suppose we are given a Riemannian covering $\pi
:(\widehat{M},\widehat{\left\langle ,\right\rangle })\rightarrow\left(
M,\left\langle ,\right\rangle \right)  $. Is there any relation between the
validity of the Feller property on the covering space $\widehat{M}$ and on the
base manifold $M$?
\end{problem}

By comparison, recall that $M$ is stochastically complete if and only if so is
$\widehat{M}$. Passing from the covering to the base is easy via the use of
bounded, $\lambda$-subharmonic functions. The converse seems to be
non-trivial. A proof using stochastic differential equations can be found in
the book by D. Elworthy, \cite{El-book}, but it would be nice to have a
deterministic proof of this fact. Intuitively, Elworthy proof relies on the
fact that (similarly to what happens for geodesics) Brownian paths in $M$ lift
to Brownian paths in $\widehat{M}$ and, conversely, Brownian paths in
$\widehat{M}$ project to Brownian paths in $M$.

As for parabolicity, the situation is quite different. Using subharmonic
functions it is easy to see that if the covering manifold $\widehat{M}$ is
parabolic then the base manifold $M$ is also parabolic. In general, the
converse is not true, as shown e.g. by the twice punctured complex plane. This
latter is a parabolic manifold which is universally covered by the
(non-parabolic) Poincar\`{e} disk.\smallskip

Let us now consider the Feller property. To begin with, consider the easiest
case of coverings with a finite number of sheets. As expected, we have the following

\begin{proposition}
\label{prop_cov1}Let \textit{ }$\pi:(\widehat{M},\widehat{\left\langle
,\right\rangle })\rightarrow\left(  M,\left\langle ,\right\rangle \right)  $
be a $k$-fold Riemannian covering, $k<+\infty$. Then, $\widehat{M}$ is Feller
if and only if $M$ is Feller.
\end{proposition}

\begin{proof}
Let $\Omega\subset\subset M$ be fixed and set $\widehat{\Omega}=\pi
^{-1}\left(  \Omega\right)  \subset\subset\widehat{M}$. Let $\lambda>0$ be a
chosen number. We are going to show that the minimal, positive $\lambda
$-harmonic function $h$ on $M\backslash\Omega$ with boundary data $\left.
h\right\vert _{\partial\Omega}=1$ \ is related to the minimal, positive
$\lambda$-harmonic function $\widehat{h}$ \textit{on }$\widehat{M}%
\backslash\widehat{\Omega}$ with boundary data $\left.  \widehat{h}\right\vert
_{\partial\widehat{\Omega}}=1$\textit{ by}%
\begin{equation}
\widehat{h}=h\circ\pi\mathit{.} \label{cov0}%
\end{equation}
Indeed, let $\Omega_{n}\nearrow M$ be a compact exhaustion and, for each $n$,
let $h_{n}$ be the solution of the Dirichlet problem%
\[
\left\{
\begin{array}
[c]{l}%
\Delta h_{n}=\lambda h_{n}\text{ on }\Omega_{n}\backslash\bar{\Omega}\\
h_{n}=1\text{ on }\partial\Omega\\
h_{n}=0\text{ on }\partial\Omega_{n}.
\end{array}
\right.
\]
Then, \textit{ }$\widehat{\Omega}_{n}=\pi^{-1}\left(  \Omega_{n}\right)
\nearrow\widehat{M}$ is a compact exhaustion and, since $\pi$ is a local
isometry,%
\[
\widehat{h}_{n}=h_{n}\circ\pi
\]
solves the analogous Dirichlet problem on $\widehat{\Omega}_{n}\backslash
\widehat{\Omega}$. The desired relation between $h$ and $\widehat{h}$ follows
by letting $n\rightarrow+\infty$.

Now, suppose $\widehat{M}$ is Feller. We show that $M$ must be Feller, i.e.,
for every $\varepsilon>0$ there exists a compact set $K\subset M$ such that,
for every $x\in M\backslash K$, $h\left(  x\right)  <\varepsilon$. Let
$\varepsilon>0$ be fixed. Then, by assumption, there is a compact $\widehat
{K}\subset\widehat{M}$ such that $\widehat{h}\left(  \widehat{x}\right)
<\varepsilon$ \ for every $\widehat{x}\in\widehat{M}\backslash\widehat{K}$.
Define $K=\pi(\widehat{K})\subset M$ and the further compact subset
$\widehat{K}_{1}=\pi^{-1}\left(  K\right)  \subset\widehat{M}$. Clearly,
$\widehat{K}\subset\widehat{K}_{1}$ so that $\widehat{h}\left(  \widehat
{x}\right)  <\varepsilon$ whenever $\widehat{x}\in\widehat{M}\backslash
\widehat{K}_{1}$. It follows from (\ref{cov0}) that, for every $x\in
M\backslash K$, $h\left(  x\right)  =\widehat{h}\left(  \widehat{x}\right)
<\varepsilon$, where $\widehat{x}\in\pi^{-1}\left(  x\right)  $ is chosen arbitrarily.

Assume now that $M$ is Feller. We show that $\widehat{M}$ is Feller. To this
end, having fixed $\varepsilon>0$, let $K\subset M$ be a large compact set
such that $h\left(  x\right)  <\varepsilon$ for every $x\in M\backslash K$.
Let us consider the compact set $\widehat{K}=\pi^{-1}\left(  K\right)
\subset\widehat{M}$. If $\widehat{x}\in\widehat{M}\backslash\widehat{K}$ then
$x=\pi\left(  \widehat{x}\right)  \in M\backslash K$ and, according to
(\ref{cov0}) we deduce $\widehat{h}\left(  \widehat{x}\right)  =h\left(
x\right)  <\varepsilon$, completing the proof.
\end{proof}

Observe that there are two key points in the above proof:

\begin{enumerate}
\item[(A)] a $k$-fold covering map is proper;

\item[(B)] for a $k$-fold covering, conditions $\widehat{x}\rightarrow
\widehat{\infty}$ and $x=\pi\left(  \widehat{x}\right)  \rightarrow\infty$ are
essentially the same.
\end{enumerate}

Obviously, the situation changes drastically if we consider an $\infty$-fold
Riemannian covering $\pi:\widehat{M}\rightarrow M$. Violating (A) yields that
the Feller property does not descend on the base manifold. In the next example
we show that%
\[
\widehat{M}\text{ Feller }\not \Longrightarrow M\text{ Feller.}%
\]

\begin{example}
\label{ex_covering1}Consider the $2$-dimensional warped product $M=\mathbb{R}%
\times_{f}\mathbb{S}^{1}$ where $f\left(  t\right)  =e^{t^{3}}$. According to
Corollary \ref{cor_warped}, we can use Example \ref{ex_versus1} to deduce that
$M$ is not Feller. Note that the Gaussian curvature of $M$ is given by%
\[
K\left(  t,\theta\right)  =-\frac{f^{\prime\prime}\left(  t\right)  }{f\left(
t\right)  }\leq0.
\]
Therefore the universal covering $\widehat{M}$ is Cartan-Hadamard, hence
Feller by Theorem \ref{th_comparison}.
\end{example}
By contrast,the reverse implication
\[
M\text{ is Feller }{\Longrightarrow}\widehat{M}\text{ is Feller}%
\]
holds, and it is the content of the main theorem of this section.

It is not obvious how to achieve the proof using the
elliptic point of view, so we adopt the heat kernel point of view.
We begin with a simple lemma that will be used in the proof. Recall
that since manifolds are second countable, $\pi_1(M)$ is necessarily
countable. Moreover the compact open topology induced by its action on
$\widehat{M}$ coincides with the discrete topology. To say that the sequence
$\gamma_k \to \infty$ in the compact open topology means that
$\gamma_k$ is eventually in the complement of any finite set.

\begin{lemma}
\label{lemma-cov} Let $\widehat{B}$ be a ball in $\widehat{M}$.
Then for every $\widehat{x}\in \widehat{M}$ and every sequence
$\{\gamma_k\}\to \infty $ in $ \pi_1(M)$
\begin{equation}
\label{lemma-cov-eq1}
\lim_{k} \int_{\widehat{B}}
\widehat{p}_t(\gamma_k\widehat{x},\widehat{y})d\widehat{y}=0.
\end{equation}
\end{lemma}

\begin{proof}
Assume first that $\widehat{B}= \widehat{B}_r(\widehat{z})$ is contained in a fundamental
domain, so that $\pi :\widehat{B}_r(\widehat{z})\to B_r(\pi(\widehat{z}))$ is an isometry
and
\begin{equation*}
\pi^{-1}(B_r(\pi(\widehat{z})) = \cup_{\gamma\in \pi_1(M)}\gamma \widehat{B}_r(\widehat{z})
\end{equation*}
is a disjoint union. It follows that for every $\widehat{x}$
\begin{equation*}
\sum_{\gamma\in \pi_1(M)} \int_{\gamma \widehat{B}_r(\widehat{z})}
\widehat{p}_t(\widehat{x}, \widehat{y}) d\widehat{y}\leq
\int_{ \widehat{M}}
\widehat{p}_t(\widehat{x}, \widehat{y}) d\widehat{y}\leq 1.
\end{equation*}
In particular,
\begin{equation*}
\int_{\gamma^{-1} \widehat{B}_r(\widehat{z})}
\widehat{p}_t(\widehat{x}, \widehat{y}) d\widehat{y}\to 0 \text{ as
} \gamma \to \infty,
\end{equation*}
and therefore
\begin{equation*}
\int_{\gamma^{-1}_k \widehat{B}_r(\widehat{z})}
\widehat{p}_t(\widehat{x}, \widehat{y}) d\widehat{y}\to 0 \text{ as
} k \to \infty.
\end{equation*}
On the other hand, since $\pi_1(M)$ acts isometrically on $\widehat
M$, for every $\gamma$,\, $\widehat{p}_t(\widehat{x}, \gamma^{-1} \widehat {y})$ solves the
heat equation and converges to $\delta_{\gamma\widehat{x}}$ as $t\to 0+$, so, by minimality,
\begin{equation*}
\widehat{p}_t (\gamma \widehat{x}, \widehat{y}) =
\widehat{p}_t (\widehat{x}, \gamma^{-1} \widehat{y}).
\end{equation*}
This, together with a change of variables gives
\begin{equation*}
\int_{ \widehat{B}_r(\widehat{z})}
\widehat{p}_t(\gamma_k\widehat{x},  \widehat{y}) d\widehat{y}
=
\int_{ \widehat{B}_r(\widehat{z})}
\widehat{p}_t(\widehat{x}, \gamma^{-1}_k \widehat{y}) d\widehat{y}
=
\int_{\gamma^{-1}_k  \widehat{B}_r(\widehat{z})}
\widehat{p}_t(\widehat{x}, \widehat{y}) d\widehat{y}.
\end{equation*}
and (\ref{lemma-cov-eq1}) is proved in this case. The case where $\widehat{B}$
is not contained in a fundamental domain is dealt with using a
standard covering argument.
\end{proof}

\begin{theorem}
\label{th_cov2} If $M$ is Feller then so is $\widehat{M}$.
\end{theorem}
\begin{proof}
We need to show that if $\widehat{B}_o$ is a ball is $\widehat{M}$,
then, for every sequence $\widehat x_k\to \infty$ in $\widehat M$,
\begin{equation*}
\int_{\widehat{B}_o}
\widehat{p}_t(\widehat{x}_k, \widehat{y}) d\widehat{y}\to 0\,\,
 \text{ as } k \to \infty.
\end{equation*}
Fix $\epsilon>0.$ According to a result of M.~Bordoni, \cite{Bo-BSMF}, for every
$0\leq\varphi\in C_{c}\left(  M\right)  $, we have
\begin{equation}
\label{cov4}
\int_{\widehat{M}}\widehat{p}_{t}\left(  \widehat{x},\widehat{y}\right)
(\varphi\circ\pi\left(\widehat{y}\right))  d\widehat{y}
\leq
\int_{M}p_{t}\left(  \pi\left(  \widehat{x}\right)  ,y\right)
\varphi\left(
y\right)  dy
\end{equation}
(and  equality holds if  $M$, or
equivalently $\widehat{M}$, is stochastically complete), so, if
$\varphi$ is a cut-off function such that $\pi(\widehat{B}_o )\subset
\{\varphi = 1\} $,
\begin{equation*}
\int_{\widehat{B}_o}\widehat{p}_{t}\left(  \widehat{x},\widehat{y}\right)
 d\widehat{y}
\leq
\int_{M}p_{t}\left(  \pi\left(  \widehat{x}\right)  ,y\right)
\varphi\left(
y\right)  dy.
\end{equation*}
Since $M$ is Feller, there exists a ball $D\subset
M$ such that, if $\pi(\widehat{x}_k)\not\in \overline D$ then the
integral on the right hand side is less than $\epsilon,$ and
therefore
\begin{equation*}
\label{cov3}
\int_{\widehat{B}_o}\widehat{p}_{t}\left(  \widehat{x}_k,\widehat{y}\right)
 d\widehat{y}
< \epsilon.
\end{equation*}
Next let $\widehat{B}$ be a ball in $\widehat{M}$ such that
$\pi(\widehat{B}) \supset \overline D$, so that $\pi^{-1}(\overline D)
\subset \cup_{\gamma\in \pi_1(M)} \gamma \widehat{B}$.

Without loss of generality, we may therefore assume that $\pi(\widehat{x}_k)\in
\overline D $ for every $k$, so that there exist  $\widehat{z}_k\in
\widehat{B}$ and a sequence $\gamma_k\in \pi_1(M)$ such that $\widehat{x}_k = \gamma_k
\widehat{z}_k$. Since $\widehat{x}_k\to \infty $ in $\widehat{M}$, $\gamma_k\to \infty $ in
$\pi_1(M)$. Moreover, since $\gamma_k$ is an isometry of $\widehat{M}$,
\begin{equation*}
(\widehat{z},t)\to\int_{\widehat{B}_o} \widehat p_t(\gamma_k \widehat{z},
\widehat{y}) d\widehat{y}
\end{equation*}
is a solution of the heat equation on $M$, by the parabolic mean
value inequality (see, e.g., \cite{Saloff-Aspects}), there exists a
constant $C(\widehat B, t)$ depending only on the geometry of $2\widehat{B}$
and $t$ such that, for every $\widehat{z}_k\in \widehat{B}$
\begin{equation*}
\int_{\widehat{B}_o} \widehat p_t(\gamma_k \widehat{z}_k, \widehat{y}) d\widehat{y}
\leq C(\widehat B, t) \int_{t/2}^t ds \int_{2\widehat{B}}
d\widehat{z'} \int_{\widehat{B}_o} \widehat p_s(\gamma_k \widehat{z'},
\widehat{y}) d\widehat{y},
\end{equation*}
and the right hand side tends to zero as $k\to \infty$ by Lemma~\ref{lemma-cov} and
the dominated convergence theorem.
\end{proof}

\medskip
We remark that the proof actually shows that if $\widehat{x}_k \to
\infty$ but $\pi(\widehat{x}_k)$ is contained in a compact set for
every $k$ then
\begin{equation*}
\int_{\widehat{B}_o} \widehat p_t(\gamma_k \widehat{z}_k, \widehat{y})
d\widehat{y} \to 0 \text{ as } k\to \infty
\end{equation*}
without having to assume that $M$ be Feller.

We also note that the proof of Proposition 2.4 in \cite{Bo-BSMF} can be easily
adapted to show that (\ref{cov4}) always holds with equality in the case of
$k$-fold coverings. This could be used to give an alternative  proof of
Proposition~\ref{prop_cov1}.

\section{On the curvature condition by E. Hsu and some remarks on the role of
volumes\label{section hsu}}

As we mentioned in the Introduction, a complete Riemannian manifold is Feller
provided a suitable control from below on its Ricci tensor is assumed. In this
direction, the best known result in the literature is the following theorem by
Hsu, \cite{Hsu-annals}, \cite{Hsu-book}, which extends previous work by Yau,
Dodziuk and Li-Karp.

\begin{theorem}
\label{th_hsu}Let $\left(  M,\left\langle ,\right\rangle \right)  $ be a
complete, non compact Riemannian manifold of dimension $\dim M=m.$ Assume that%
\begin{equation}
^{M}Ric\geq-\left(  m-1\right)  G^{2}\left(  r\left(  x\right)  \right)  ,
\label{hsu1}%
\end{equation}
where $r\left(  x\right)  =dist\left(  x,o\right)  $ is the distance function
from a fixed reference point $o\in M$ and $G$ is a positive, increasing
function on $[0,+\infty)$ satisfying%
\begin{equation}
\frac{1}{G}\notin L^{1}\left(  +\infty\right)  . \label{hsu2}%
\end{equation}
Then $M$ is Feller.
\end{theorem}

\begin{remark}
Unlike a similar result for the validity of the stochastic completeness, this
theorem is not a comparison-type theorem. As indicated in Section
\ref{section comparison}, curvature comparisons should go exactly in the
opposite direction (the same direction of heat kernel comparisons). Hsu
theorem, like its \textquotedblleft predecessors\textquotedblright, is a
genuine estimating result. The proof supplied by Hsu is very probabilistic in
nature. To the best of our knowledge there is no deterministic proof (neither
for general manifolds nor for the easiest case of models) and we feel its
discovery would be very interesting.
\end{remark}

According to Theorem \ref{th_hsu}, $M$ is Feller provided its Ricci curvature
does not decay to $-\infty$ too much quickly. To fix ideas one may think of
$G\left(  t\right)  $ as the function $t\Pi_{j=1}^{n}\log^{\left(  j\right)
}\left(  t\right)  $, where $\log^{\left(  j\right)  }\left(  t\right)  $
denotes the $j$-th iterated logarithm and $n\in\mathbb{N}$ is arbitrarily
large. Using the results of Section \ref{section models} we are able to prove
that such curvature condition is, in some sense, sharp.

\begin{example}
\label{ex_hsu1}Let $G:\mathbb{R}\rightarrow\mathbb{R}$ be a smooth, increasing
function satisfying%
\begin{equation}
\text{(a) }G\left(  r\right)  >0\text{, (b) }\limsup_{r\rightarrow+\infty
}\frac{G^{^{\prime}}\left(  r\right)  }{G\left(  r\right)  ^{2}}%
=\alpha<+\infty\text{, (c) }\frac{1}{G\left(  r\right)  }\in L^{1}\left(
+\infty\right)  . \label{hsu3}%
\end{equation}
Fix $\beta>\alpha$ and let $g\left(  t\right)  :\mathbb{R}\rightarrow
\mathbb{R}$ be any smooth, positive, odd function such that $g^{\prime}\left(
0\right)  =1$ and $g\left(  r\right)  =\exp\left(  -\beta\int_{0}^{r}G\left(
t\right)  dt\right)  $ for $r\geq10$. Let us consider the $m$-dimensional
model manifold $M_{g}^{m}$ with warping function $g$. Direct computations show
that the radial sectional curvature of $M_{g}^{m}$ satisfies%
\begin{equation}
K_{rad}\left(  r\right)  =-\frac{g^{\prime\prime}\left(  r\right)  }{g\left(
r\right)  }\leq-\beta\left(  \beta-\alpha\right)  G\left(  r\right)
^{2}\text{, }r>>1. \label{hsu4}%
\end{equation}
We claim the validity of the following conditions%
\begin{equation}
\text{(a) }g\left(  r\right)  ^{m-1}\in L^{1}\left(  +\infty\right)  \text{,
(b) }\frac{\int_{r}^{+\infty}g\left(  t\right)  ^{m-1}dt}{g\left(  r\right)
^{m-1}}\in L^{1}\left(  +\infty\right)  . \label{hsu5}%
\end{equation}
Indeed, since $G$ is positive and increasing,%
\[
g\left(  r\right)  ^{m-1}\leq\exp\left(  -\left(  m-1\right)  \beta G\left(
0\right)  r\right)  ,\text{ }r>>1,
\]
thus proving (\ref{hsu5}) (a). On the other hand, note that, for $r>>1$,%
\begin{align*}
G\left(  r\right)  \int_{r}^{+\infty}g\left(  t\right)  ^{m-1}dt  &  \leq
\int_{r}^{+\infty}G\left(  t\right)  \exp\left(  -\left(  m-1\right)
\beta\int_{0}^{r}G\left(  t\right)  dt\right) \\
&  =\frac{1}{\left(  m-1\right)  \beta}\exp\left(  -\left(  m-1\right)
\beta\int_{0}^{r}G\left(  t\right)  dt\right)  ,
\end{align*}
so that%
\[
G\left(  r\right)  \int_{r}^{+\infty}g\left(  t\right)  ^{m-1}dt\rightarrow
0\text{, as }r\rightarrow+\infty.
\]
Therefore, de l'Hospital rule applies and gives%
\begin{align*}
\limsup_{r\rightarrow+\infty}\frac{G\left(  r\right)  \int_{r}^{+\infty
}g\left(  t\right)  ^{m-1}dt}{g\left(  r\right)  ^{m-1}}  &  \leq
\limsup_{r\rightarrow+\infty}\frac{G^{\prime}\left(  r\right)  \int
_{r}^{+\infty}g\left(  t\right)  ^{m-1}dt-G\left(  r\right)  g\left(
r\right)  ^{m-1}}{\left(  m-1\right)  g\left(  r\right)  ^{m-2}g^{\prime
}\left(  r\right)  }\\
&  \leq\limsup_{r\rightarrow+\infty}\frac{-G\left(  r\right)  g\left(
r\right)  ^{m-1}}{\left(  m-1\right)  g\left(  r\right)  ^{m-2}g^{\prime
}\left(  r\right)  }\\
&  =\frac{1}{\beta\left(  m-1\right)  }%
\end{align*}
which implies%
\[
\frac{\int_{r}^{+\infty}g\left(  t\right)  ^{m-1}dt}{g\left(  r\right)
^{m-1}}\leq\frac{1}{\beta\left(  m-1\right)  }\frac{1}{G\left(  r\right)  }\in
L^{1}\left(  +\infty\right)  .
\]
This shows the validity of (\ref{hsu5}) (b).

Now, condition (\ref{hsu5}) (a) implies $1/g^{m-1}\notin L^{1}\left(
+\infty\right)  $. Therefore, from (\ref{hsu5}) (b) and applying Theorem
\ref{th_model} we conclude that $M_{g}^{m}$ \ is not Feller.
\end{example}

According to Theorem \ref{th_hsu}, and in view of the above example, the
search of more general (or even new) conditions ensuring the validity of the
Feller property should not involve pointwise curvature lower bounds. For
instance, variations on the theme could be obtained using integral curvature
bounds. More importantly, one is naturally led to ask whether a solely volume
growth condition suffices. In this respect, we quote the following intriguing
question addressed by Li and Karp, \cite{LK}.

\begin{problem}
\label{problem_LK}Let $\left(  M,\left\langle \,,\right\rangle \right)  $ be a
complete Riemannian manifold. Assume that, for some reference origin $o\in M$,%
\begin{equation}
\mathrm{vol}\left(  B_{R+1}\left(  o\right)  \backslash B_{R}\left(  o\right)
\right)  \geq e^{-AR^{2}}\text{,} \label{LK1}%
\end{equation}
for every $R>>1$ and for some constant $A>0$. Does $M$ satisfy the Feller property?
\end{problem}

Actually, as a consequence of Corollary \ref{cor_warped}, it seems that the
volume growth (decay) of a general complete manifold is not so tightly related
to the validity of the Feller property. Indeed, one can always take
$\mathbb{R}\times_{f}\mathbb{S}^{m-1}$ and prescribe the asymptotic behavior
of $f\left(  t\right)  $ at $-\infty$ and $+\infty$ in such a way that the
volume growth is slow (even finite) or fast but at least one of the ends is
not Feller. This suggests that possible conditions on volumes, such as those
specified in Problem \ref{problem_LK}, should be localized on each of the ends
of the manifold.

\begin{problem}
\label{problem_LK2}Let $\left(  M,\left\langle \,,\right\rangle \right)  $ be
a complete Riemannian manifold \ which is connected at infinity and satisfies
(\ref{LK1}). Is $M$ Feller?
\end{problem}

To the best of our knowledge, only specific examples are used to conjecture a
positive answer to this question, \cite{LK}. For instance, it is reasonable to
approach the problem by first assuming that $M$ has only one end which is a
cylindrical end, namely $E=\left(  0,+\infty\right)  \times_{f}\Sigma$ \ for
some compact manifold $\Sigma$. However, so far, even in the easiest case
$\Sigma=\mathbb{S}^{m-1}$ it is unknown whether condition (\ref{LK1}) implies
the validity of the Feller property.

\end{document}